\documentclass[11pt]{amsart}

\usepackage{amsmath,amsthm, amscd, amssymb, amsfonts}
\usepackage{epsfig,pb-diagram}
\input xy
\xyoption{all}

\usepackage[usenames, dvipsnames]{xcolor}

\newcommand{\esn}{\normalcolor{}}

\newcommand{\otb}{{\overline{\otimes}}}
\newcommand{\otk}{{\otimes_{\ku}}}
\newcommand{\Mo}{{\mathcal M}}
\newcommand{\No}{{\mathcal N}}

\newcommand{\lto}{{\longrightarrow}}
\newcommand{\tb}{\overline{T}}
\newcommand{\xib}{\overline{\xi}}
\newcommand{\KER}{\mathfrak{Ker}}
\newcommand{\Ss}{{\mathcal S}}
\newcommand{\ot}{{\otimes}}

\newcommand{\ele}{{\mathcal L}}

\newcommand{\ca}{{\mathcal C}}

\newcommand{\Do}{{\mathcal D}}

\newcommand{\Bc}{{\mathcal B}}

\newcommand{\Fc}{{\mathcal F}}

\newcommand{\op}{\rm{rev}}
\newcommand{\opp}{\rm{op}}

\newcommand{\caT}{\ca \rtimes T}
\newcommand{\catop}{\ca^{\op} \rtimes \tb}

\newcommand{\ku}{{\Bbbk}}
\newcommand{\Z}{{\mathbb Z}}

\newcommand{\uno}{ \mathbf{1}}

\newcommand{\cab}{{\underline{\ca}}}

\newcommand{\Tt}{{\mathfrak T}}

\newcommand{\id}{\mbox{\rm id\,}}

\newcommand{\eqmod}{\mbox{\rm EqMod\,}}

\newcommand{\Res}{\mbox{\rm Res\,}}
\newcommand{\Ind}{\mbox{\rm Ind\,}}
\newcommand\rrep{\mbox{-}\operatorname{mod}}
\newcommand\rep{\operatorname{mod}\!\mbox{-}}
\newcommand\corep{\operatorname{comod}\!\mbox{-}}

\newcommand{\vect}{\mbox{\rm vect\,}}

\newcommand{\Mod}{\mbox{\rm Mod\,}}

\newcommand{\Fun}{\operatorname{Hom}}
\newcommand{\Funl}{\operatorname{Hom^{lax}}}

\newcommand{\Rep}{\operatorname{Rep}}

\newcommand\Hom{\operatorname{Hom}}

\newcommand{\End}{\operatorname{End}}

\newcommand{\Id}{\mathop{\rm Id}}

\renewcommand{\_}[1]{\mbox{$_{\left( #1 \right)}$}}

\theoremstyle{plain}
\numberwithin{equation}{section}

\newtheorem{teo}{Theorem}[section]
\newtheorem{lema}[teo]{Lemma}
\newtheorem{cor}[teo]{Corollary}
\newtheorem{prop}[teo]{Proposition}

\theoremstyle{definition}
\newtheorem{defi}[teo]{Definition}
  \newtheorem{exa}[teo]{Example}

\theoremstyle{remark}

\newtheorem{rmk}[teo]{Remark}

\def\pf{\begin{proof}}
\def\epf{\end{proof}}

\theoremstyle{remark}

\begin{document}

\title[Module categories over equivariantized tensor
categories]{Module categories over equivariantized tensor categories }
\author[Mombelli and Natale]{
Mart\'\i n Mombelli and Sonia Natale}
\thanks{The work of M. M. was partially supported by
 CONICET, Secyt-UNC, Mincyt (C\'ordoba) Argentina. The work of S. N. was partially supported by CONICET, Secyt-UNC and the Alexander von Humboldt Foundation}
\address{Facultad de Matem\'atica, Astronom\'\i a y F\'\i sica
\newline \indent
Universidad Nacional de C\'ordoba
\newline
\indent CIEM -- CONICET
\newline \indent Medina Allende s/n
\newline
\indent (5000) Ciudad Universitaria, C\'ordoba, Argentina} \email{
mombelli@famaf.unc.edu.ar} \email{ natale@famaf.unc.edu.ar
\newline \indent \emph{URL:}\/
http://www.famaf.unc.edu.ar/$\sim$mombelli/welcome.html
\newline \indent \emph{URL:}\/ http://www.mate.uncor.edu/$\sim$natale }

\begin{abstract} For a  finite  tensor category $\ca$ and a Hopf monad
$T:\ca\to \ca$ satisfying certain conditions we describe exact indecomposable
left $\ca^T$-module categories in
terms of left $\ca$-module categories and some extra data. We also give a
2-categorical interpretation of the
process of equivariantization of module categories.
\end{abstract}


\date{May 30, 2014}
\maketitle

\section*{Introduction}

As is the case in the study of any algebraic structure, a fundamental r\^ ole in the study of tensor categories is played by its "representations". 
The natural notion of representation of a tensor category $\ca$ is that of a \emph{module category} over $\ca$. A (left) module category over a tensor
category $\ca$ is a $\ku$-linear Abelian category $\Mo$ equipped with a $\ca$-action, that is, an exact 
bifunctor $\otb: \ca \times \Mo \to \Mo$ endowed with functorial associativity and unit constraints which satisfy appropriate coherence conditions. 
This notion is recalled in Section \ref{s:modcat}, it can be regarded as a "categorification" of the notion of module over an algebra. 
Many papers have been devoted to the study of different aspects of module categories over a monoidal or tensor category in the last years. 

In the context of finite tensor categories it is convenient to restrict the attention to the class of \emph{exact} module categories: this class of module categories was introduced in \cite{eo}, see also \cite[Section 2.6]{EGNO}.   By definition, a module category $\Mo$ is exact if it is finite and for any
projective object $P$ of $\ca$ and for any object $M$ of $\Mo$, the object $P \otb M$ is projective.

\medbreak Examples of finite tensor categories over $\ku$ are given by the categories of finite dimensional (co)modules over a finite dimensional Hopf algebra $H$ over $\ku$. Module categories over such tensor categories have been investigated intensively for several different classes of Hopf algebras.

A natural generalization of a Hopf algebra is given by a Hopf monad, as introduced in \cite{BV}, \cite{blv}. Let $\ca$ be a tensor category over $\ku$. A    Hopf monad on $\ca$ is a monad $T$ on $\ca$ which is a comonoidal functor in a compatible way and such that certain associated fusion operators are invertible. 
If $T$ is a $\ku$-linear right exact Hopf monad on a (finite) tensor category $\ca$, then the Eilenberg-Moore category $\ca^T$ of $T$-modules in $\ca$ is also a (finite) tensor category over $\ku$ and the forgetful functor $\Fc: \ca^T \to \ca$ is a tensor functor. This functor is in addition dominant if $T$ is a faithful endofunctor of $\ca$. 

\medbreak The main goal of this paper is to give a description of exact indecomposable module categories over the tensor category $\ca^T$ of $T$-modules in a (finite) tensor category $\ca$, where $T$ is a $\ku$-linear right exact faithful Hopf monad on $\ca$. 

In order to do this we introduce the notion of a \emph{$T$-equivariant $\ca$-module category}: this consists of the data $(\Mo,
U, c)$, where $\Mo$ is a $\ca$-module category, $U$ is a monad on $\Mo$, and $(U, c): \Mo \to \Mo(T)$ is a lax $\ca$-module functor, 
such that the multiplication and unit morphisms of $U$ are morphisms of $\ca$-module functors. See Definition \ref{defi-temc}.     Here $\Mo(T)$ is a natural lax $\ca$-module category arising from $\Mo$ and the lax comonoidal functor $T$.

\medbreak We show in Theorem \ref{U-equivariant modcat} that if $\Mo$ is a $T$-equivariant
$\ca$-module category then the category $\Mo^U$ is a  $\ca^T$-module category.    We also establish some functorial properties of this assignment and, in particular,  give conditions in order that $\Mo^U$ be a simple module category in terms of $\Mo$ and $U$.

Our main result states that if $T$ is a  right exact faithful Hopf monad  on 
$\ca$ and  $\Mo$ is an exact indecomposable $\ca^T$-module category, then
there exists a $T$-equivariant indecomposable exact $\ca$-module category $\No$
with simple and exact equivariant structure $U:\No\to \No$ such that
$\Mo\simeq \No^U$ as $\ca^T$-module categories. See Theorem \ref{equi-mc} and Corollary \ref{modc-eq}.   This result can be thought of as an extension of some of the results obtained in the context of module categories over representations of finite dimensional Hopf algebras in \cite{AM} (see Example \ref{exa:mc-hopf}).

One of the tools in the proof of the main result is an investigation of the relation between module categories over the category $\caT = \End_{\ca^T}(\ca)$ of $\ca^T$-module endofunctors of $\ca$ and $T$-equivariant $\ca$-module categories. We show in Theorem \ref{equi-mc} that every $\caT$-module category $\No$ has a natural structure of a $T$-equivariant module category; in fact, the Hopf monad $T$ can be regarded as an algebra in $\caT$, and  the relevant data $U: \No \to \No(T)$ for the $T$-equivariance of $\No$ is provided by the action of $T$ on $\No$.

\medbreak Recall that, for a given tensor category $\ca$, $\ca$-module categories, (lax) $\ca$-module functors and $\ca$-module natural transformations constitute a $2$-category, that we denote ${}_\ca\Mod$ (respectively, ${}_\ca\Mod^{lax}$). We show that the assignment $\Mo \mapsto \Mo(T)$ extends to a 2-monad $\Tt$ on ${}_\ca\Mod^{lax}$, and there is a 2-equivalence of 2-categories
$${}_\ca\eqmod \simeq   ({}_\ca\Mod^{lax})^\Tt,$$ where ${}_\ca\eqmod$ is the 2-category of $T$-equivariant lax
$\ca$-module categories. This is proved in Proposition \ref{equiv-eqmod}, and gives a 2-categorical interpretation of  the
process of equivariantization of module categories.

\medbreak 
As an application we give a description of module categories over Hopf algebroids, as defined for instance in \cite{Bo},
\cite{BS}, \cite{KS}. We show in Theorem \ref{mc-halgd} that under the assumption that the basis of the Hopf algebroid $H$ is simple  (which guarantees that $H\rrep$ is indeed a tensor category), then every exact indecomposable module category over $H\rrep$ is equivalent to ${}_K\Mo$ for some
$H$-simple left $H$-comodule algebra $K$.

\medbreak We then consider the special situation where the Hopf monad $T$ is \emph{normal}, according to the definition given in \cite{BN}: recall that this means that $T$ restricts to a Hopf monad on the trivial subcategory of $\ca$. Such Hopf monad gives rise to an exact sequence of tensor categories 
$\corep  H \longrightarrow \ca^T \longrightarrow \ca$, where $H$ is the induced Hopf algebra of $T$, which is finite dimensional.   
In this context we study the category $\caT$ and show that it is (reversed) equivalent as a $\ku$-linear category to the Deligne tensor product $H\rrep \boxtimes \, \ca$. 

\medbreak The paper is organized as follows. In Sections \ref{pr-nt} and \ref{s:modcat} we recall the definitions and main basic features of tensor categories and their module categories and Hopf monads on tensor categories and the associated categories of modules, respectively. In particular, given a Hopf monad $T$ on a tensor category $\ca$, we discuss in this section the Morita dual of the category $\ca^T$ with respect to its canonical module category $\ca$. In Section \ref{mc-cat} we study module categories over the category $\ca^T$. Theorem \ref{equi-mc} and Corollary \ref{modc-eq} are proved in this section; the notions of $T$-equivariant module category and simple $T$-equivariant module category are also introduced here. Section \ref{h-algebroids} presents an application of the results in the previous section to the category of representations of a finite-dimensional Hopf algebroid. In Section \ref{2-cat} we give a 2-categorical interpretation of equivariantization of module categories. Finally in 
Section \ref{mc-exact} we discuss the case where the Hopf monad $T$ is normal and give some examples. 

\subsection*{Acknowledgement} 
The authors thank I. Lopez Franco for answering some questions on 2-categories. Remark \ref{ilf} is due to him. 
The work of S. Natale was done partly during a 
research stay in the University of Hamburg; she thanks the Humboldt
Foundation, C. Schweigert and the Mathematics Department of U. Hamburg for the
kind hospitality.

\section{Preliminaries and Notation}\label{pr-nt}

We shall work over an algebraically closed field $\ku$ of characteristic 0.
All vector spaces and algebras will be over $\ku$. If $A$ is an algebra then
${}_A\Mo$ will
denote the category of finite-dimensional left $A$-modules.  
If $H$ is a Hopf algebra, we shall denote by $H\rrep$, respectively $\rep H$,
 the category of finite-dimensional left (respectively right) $H$-modules and by
$ \corep H$ the category of 
finite-dimensional left $H$-comodules.

\subsection{Tensor categories}

A \emph{tensor category over} $\ku$ is a $\ku$-linear Abelian rigid monoidal
category $\ca$  such that the tensor product functor $\otimes : \ca \times \ca
\to \ca$ is $\ku$-linear in each variable, and the following conditions hold:
\begin{itemize}\item Hom spaces are finite dimensional,
\item all objects of $\ca$ have finite length,
\item the unit object $\bf 1$ is simple.
\end{itemize}

A \emph{finite tensor category} \cite{eo} is a tensor category  that  has a
finite number of isomorphism classes
 of simple objects  and
every simple object has a projective cover.
Hereafter all tensor categories will be considered over $\ku$ and every functor
will be assumed to be $\ku$-linear.

 Observe that if $\ca$ is a tensor category over $\ku$,
then it follows by rigidity that the tensor product functor
$\otimes : \ca \times \ca \to \ca$ is bi-exact.

\medbreak If $\ca$ is a tensor category, we shall denote by $\ca^{\op}$ the
tensor category whose
underlying Abelian  category is $\ca$, endowed with the opposite tensor product:
$$X\ot^{\op} Y= Y\ot X, \quad X, Y\in \ca.$$
and associativity constraint
$a^{\op}_{X,Y,Z}=a^{-1}_{Z,Y,X}$, $ X, Y, Z\in \ca$.
 Throughout this paper all tensor
categories will be assumed to be strict, unless explicitly mentioned.

\subsection{Tensor functors}\label{tensor-functor}   A \emph{tensor
functor} from a tensor category $\ca$ to a
tensor category $\Do$ is a $\ku$-linear  exact strong monoidal functor
$F: \ca \to \Do$.   A tensor functor preserves
duals and is automatically  faithful.

A tensor functor $F : \ca \to \mathcal D$  is called
\emph{dominant} if it satisfies any of the following equivalent
conditions (\cite[Lemma 3.1]{BN}):
\begin{enumerate}
 \item[(i)] Any object  $Y$ of $\Do$ is a subobject of $F(X)$ for some object
$X$ of $\ca$;
 \item[(ii)]  Any object  $Y$ of $\Do$ is a quotient of $F(X)$ for some object
$X$ of $\ca$;
 \item[(iii)] The Pro-adjoint of $F$ is faithful;
 \item[(iv)] The Ind-adjoint of $F$ is faithful.
\end{enumerate}
On the other hand, $F$ is called \emph{surjective} if any object of
$\Do$ is a subquotient of $F(X)$ for some $X$ of $\ca$
\cite[Definition 2.4]{eo}.  In particular, every  dominant tensor functor is
surjective.

\medbreak Let $F : \ca \to \mathcal D$ be a tensor functor between tensor
categories $\ca, \Do$. Suppose that $F$ admits a right adjoint $R: \Do \to \ca$.
Note that $R$ is automatically faithful since $F$ is dominant.
By \cite[Proposition 6.1]{BN}, $A = R(\uno)$ has a structure of  a
central commutative algebra $(A, \sigma)$ in $\mathcal Z(\ca)$.

\medbreak Assume in addition that the right adjoint $R: \Do \to \ca$ of $F$ is
exact. In this case, $F$ is called a \emph{perfect} tensor functor. Then the
category $\ca_A$ of right $A$-modules in $\ca$ is a tensor category with the
monoidal structure induced by $\otimes_A$ and the half-braiding
$\sigma$, and the functor $F$ is equivalent over $\ca_A$ to the
free module functor $F_A: \ca \to \ca_A$, $X \mapsto X \otimes A$.
That is, there is an equivalence of tensor categories $K: \Do \to \ca_A$ such
that $KF = F_A$.

\medbreak Suppose $\ca$ and $\Do$ are finite tensor categories and let $F : \ca
\to \mathcal D$ be a dominant tensor functor. Then $F$ admits
(left and right) adjoints. Furthermore, if $\Do$ is a fusion category then the
right adjoint of $F$ is exact and therefore $F$ is a perfect tensor functor.    
See \cite[Subsection 2.2]{indp-exact}.

\section{Module categories}\label{s:modcat}

 A (left) \emph{module category} over a tensor
category $\ca$ is a  locally finite   $\ku$-linear  Abelian category $\Mo$ equipped with a 
bifunctor $\otb: \ca \times \Mo \to \Mo$, that we will sometimes refer as the
\emph{action},  which is $\ku$-bilinear and bi-exact,  endowed with 
 natural associativity
and unit isomorphisms $m_{X,Y,M}: (X\otimes Y)\otb M \to X \otb
(Y\otb M)$, $\ell_M: \uno \otb M\to M$. These isomorphisms are subject to the following conditions:
\begin{equation}\label{left-modulecat1} m_{X, Y, Z\otimes M}\; m_{X\otimes Y, Z,
M}= (\id_{X}\otimes m_{Y,Z, M})\;  m_{X, Y\otimes Z, M},
\end{equation}
\begin{equation}\label{left-modulecat2} (\id_{X}\otimes l_M)m_{X,{\bf
1} ,M}= \id_{X \otb M}.
\end{equation}
 See \cite[Subsection 2.3]{EGNO}.  Sometimes we shall also say  that $\Mo$ is a $\ca$-\emph{module}.
\medbreak

We shall say that $\Mo$ is a \emph{lax} $\ca$-module when possibly
 the associativity and unit maps $m_{X,Y,M}$ and $\ell_M$  are not necessarily
isomorphisms.

\medbreak

A \emph{module functor} between module categories $\Mo$ and $\Mo'$ over a
tensor category $\ca$ is a pair $(F,c)$, where
\begin{enumerate}
 \item[$\bullet$] $F:\Mo \to
\Mo'$ is a  left exact  functor;

\item[$\bullet$]  $c$ is a natural isomorphism: $c_{X,M}: F(X\otb M)\to
X\otb F(M)$, $X\in  \ca$, $M\in \Mo$,  such that
for any $X, Y\in
\ca$, $M\in \Mo$:
\begin{align}\label{modfunctor1}
(\id_X \otb  c_{Y,M})c_{X,Y\otb M}F(m_{X,Y,M}) &=
m_{X,Y,F(M)}\, c_{X\otimes Y,M}
\\\label{modfunctor2}
\ell_{F(M)} \,c_{\uno ,M} &=F(\ell_{M}).
\end{align}
\end{enumerate}
 If the maps $c_{X,M}$ satisfying \eqref{modfunctor1}
 and \eqref{modfunctor2} are not necessarily isomorphisms,
the pair $(F,c)$ will be called a \emph{lax module functor}.

\medbreak
There is a composition
of module functors: if $\Mo''$ is another module category and
$(G,d): \Mo' \to \Mo''$ is another module functor then the
composition
\begin{equation}\label{modfunctor-comp}
(G\circ F, e): \Mo \to \Mo'', \qquad  e_{X,M} = d_{X,F(M)}\circ
G(c_{X,M}),
\end{equation} is
also a module functor.

\smallbreak  Let $\Mo_1$ and $\Mo_2$ be $\ca$-modules.
We denote by $\Hom_{\ca}(\Mo_1, \Mo_2)$ the category whose
objects are module functors $(F, c)$ from $\Mo_1$ to $\Mo_2$. A
morphism between  $(F,c)$ and $(G,d)\in\Hom_{\ca}(\Mo_1,
\Mo_2)$ is a natural transformation $\alpha: F \to G$ such
that for any $X\in \ca$, $M\in \Mo_1$:
\begin{gather}
\label{modfunctor3} d_{X,M}\alpha_{X\otb M} =
(\id_{X}\otb \alpha_{M})c_{X,M}.
\end{gather}
We shall also say that $\alpha: F \to G$ is a $\ca$-\emph{module
transformation}.

\smallbreak
Two module categories $\Mo_1$ and $\Mo_2$ over $\ca$
are {\em equivalent} if there exist module functors $F:\Mo_1\to
\Mo_2$ and $G:\Mo_2\to \Mo_1$ and natural isomorphisms
$\id_{\Mo_1} \to F\circ G$, $\id_{\Mo_2} \to G\circ F$ that
satisfy \eqref{modfunctor3}.

\medbreak
The  direct sum of two module categories $\Mo_1$ and $\Mo_2$ over
a tensor category $\ca$  is the $\ku$-linear category $\Mo_1\times
\Mo_2$ with coordinate-wise module structure. A module category is
{\em indecomposable} if it is not equivalent to a direct sum of
two non trivial module categories.

\medbreak Let $\ca$ be a finite tensor category. Recall from \cite{eo} that  a
module category $\Mo$ is \emph{exact} if $\Mo$ is finite and for any
projective object
$P\in \ca$ the object $P\otb M$ is projective in $\Mo$, for all
$M\in\Mo$.

\begin{exa}\label{dominant-exact} Let $F: \ca \to \Do$ be a dominant tensor
functor between finite tensor categories $\ca, \Do$. Then the functor
$\otb: \ca \times \Do \to \Do$, given by $X \otb  Y = F(X)
\ot Y$, for all $X \in \ca$, $Y\in \Do$, endows  $\Do$ with a
structure of an indecomposable $\ca$-module category.
Since $F$ is dominant (thus surjective), then  $\Do$ is in fact an exact module
category over $\ca$; see \cite[Example 3.3 (i)]{eo}. \end{exa}

A submodule category of a $\ca$-module $\Mo$ is a Serre subcategory
$\No$ such that the inclusion functor $\No\to \Mo$ is a module functor.
A module category is \emph{simple} if it has no non-trivial submodule
categories.
It is known that for exact module categories the notions of indecomposability
and
simplicity are equivalent.

\begin{rmk}\label{dual-mc} If $\ca$ is a finite tensor category and $\Mo$ is an
indecomposable exact
$\ca$-module, the dual category $\ca^*_{\Mo}=\End_\ca(\Mo)$ is again
a finite tensor category \cite{eo}. It is shown in \cite[Theorem 3.31]{eo} that
there is a bijective
correspondence between equivalence classes of exact  indecomposable
left  module categories over $\ca$ and over $\ca^*_{\Mo}$.
The correspondence
assigns to a left $\ca$-module category $\No$ the
 left $\ca^*_{\Mo}$-module category $\Hom_\ca(\No, \Mo)$. This fact implies
that there is a bijective
correspondence between  equivalence classes of  exact indecomposable
left $\ca$-module categories and  equivalence classes of  exact indecomposable
\emph{right}  $\ca^*_{\Mo}$-module categories, which assigns to every
 left $\ca$-module category $\No$ the right $\ca^*_{\Mo}$-module category
$\Hom_\ca(\Mo ,\No )$. \end{rmk}

 Let  $(F, \xi,\phi):\ca\to  \tilde \ca$ be a comonoidal functor and let $(\Mo,
\otb, m)$ be a module
category over $\tilde \ca$. We shall denote by $\Mo(F)$ the lax module category
over $\ca$  with underlying
Abelian category $\Mo$   and action, associativity and unit  morphisms  defined,
 respectively,  by
$$X\otb^F M=F(X)\otb M,$$ $$m_{X,Y,M}^F=m_{F(X),F(Y),M} (\xi_{X,Y}\otb\,
\id_M), \quad l^F_M= l_M (\phi\otb \id_M),$$
for all $X, Y\in \ca$, $M\in \Mo$.

\section{Hopf monads and tensor categories}\label{hm-tc}

\subsection{Hopf Monads}\label{h-monads} Let $\ca$ be a  category.
A \emph{monad} on $\ca$ is an algebra in the strict  monoidal  category
$\End(\ca)$, that is, a triple $(T,\mu,\eta)$ where $T:\ca\to \ca$ is a functor,
$\mu:T^2\to T$ and $\eta:\Id\to T$ are natural transformations such that
\begin{align}\label{monads} \mu_X T(\mu_X)=\mu_X\mu_{T(X)}, \quad  \quad
\mu_X\eta_{T(X)}=\id_{T(X)}=\mu_X T(\eta_X).
\end{align}

Let $(T,\mu,\eta)$ be a monad on a category $\ca$. An \emph{action} of $T$ on an
object $X$ of
$\ca$ is a morphism
$r:T(X) \to X$ in $\ca$ such that:
\begin{equation}\label{module:monads}
r T(r)= r \mu_X \quad \text{and} \quad r \eta_X= \id_X.
\end{equation}
The pair $(X,r)$ is  called a \emph{$T$-module}.
Given two $T$-modules $(X,r)$ and $(Y,s)$ in $\ca$, a morphism of
$T$-modules from $(X,r)$ to $(Y,s)$ is a
morphism $f\in \Hom_\ca(X,Y)$  such that $f\circ r=s\circ T(f)$.
 The category of $T$-modules will be denoted by $\ca^T$. We shall denote by 
$\Fc = \Fc_T:\ca^T \to \ca$ the forgetful functor defined by $\Fc(X, r) = X$.

\medbreak  Let $(T_1,\mu_1,\eta_1)$, $(T_2,\mu_2,\eta_2)$ be monads on $\ca$.
 A \emph{morphism of monads} $\alpha:(T_1,\mu_1,\eta_1) \to (T_2,\mu_2,\eta_2)$
 (below indicated by $\alpha: T_1 \to T_2$) is a natural transformation $\alpha:
T_1 \to T_2$
such that
$$\alpha_X{\mu_1}_X = {\mu_2}_X  \alpha_{T_2(X)} T(\alpha_X),
\qquad \alpha_X{\eta_1}_X ={\eta_2}_X, \qquad \text{for all } X \in \ca.$$

In view of \cite[Lemma 1.7]{BV}, a morphism of monads $\alpha:T_1 \to T_2$
induces a functor $\alpha^*:\ca^{T_2} \to \ca^{T_1}$ over $\ca$, in the form
$\alpha^*(X, r) = (X, r \circ \alpha_X)$, such that $\Fc_{T_1} \alpha^* =
\Fc_{T_2}$. Furthermore, every such functor is of the form $\alpha^*$ for some
morphism of monads $\alpha:T_1 \to T_2$. \esn

\medbreak A \emph{bimonad} on a monoidal category $\ca$ is a monad
$(T,\mu,\eta)$ on
 $\ca$  such that the functor $T$ is equipped with a comonoidal structure and
the natural transformations $\mu$ and $\eta$ are comonoidal
transformations.
This means that  there is a natural transformation
$\xi_{X,Y}:T(X\ot Y)\to T(X)\ot T(Y)$ and a morphism $\phi:T(\uno)\to \uno$ in $\ca$ such
that the following conditions hold:
\begin{equation}\label{hopfmonad1}  (\id_{T(X)}\ot \xi_{Y,Z}) \xi_{X,Y\ot Z}
= (\xi_{X,Y}\ot\id_{T(Z)}) \xi_{X\ot Y, Z},
\end{equation}
\begin{equation}\label{hopfmonad2} (\id_{T(X)}\ot\phi ) \xi_{X,\uno}
=\id_{T(X)}= (\phi\ot \id_{T(X)})\xi_{\uno, X},
\end{equation}
\begin{equation}\label{hopfmonad3} \xi_{X,Y} \mu_{X\ot Y}= (\mu_{X}\ot \mu_{Y})
\xi_{T(X),T(Y)} T(\xi_{X,Y}),
\end{equation}
\begin{equation}\label{hopfmonad11} \phi\mu_\uno= \phi T(\phi), \quad
\xi_{X,Y}\eta_{X\ot Y}=\eta_X\ot\eta_Y, \quad\phi\eta_\uno=\id_\uno.
\end{equation}

\begin{rmk} It is not required that $T$ is a strong  comonoidal  functor, meaning
that $\xi_{X,Y}$  might  not be  isomorphisms.
\end{rmk}

If $T$ is a bimonad on the monoidal category $\ca$, then $\ca^T$ is a monoidal
category with tensor product
$$(X,r)\ot (Y,s)= (X\ot Y, (r\ot s)\xi_{X,Y}),$$
for all $ (X,r),(Y,s) \in \ca^T$. The  unit  object of $\ca^T$ is  $(\uno, \phi)$.
 For more details see \cite{Moe}, \cite{BV}.

 Note that in this case the forgetful functor $\Fc: \ca^T\to \ca$ is a strict
strong monoidal functor.
The functor $\Fc$ has a left adjoint $\ele: \ca \to \ca^T$, such that $\ele(X) =
(T(X), \mu_X)$,
for every object $X$ of $\ca$. The  unit and   counit of the adjunction $(\ele,
\Fc)$  are given, respectively, by $\eta_X: X \to T(X)$ and
   $\epsilon_{(M, r)} = r: (T(M), \mu_M) \to (M, r)$, for all  $X \in \ca$, 
$(M, r) \in \ca^T$.

The left adjoint $\ele$ is a comonoidal functor with comonoidal structure given
by
\begin{equation}  \ele_2(X, Y) =  \xi_{X, Y}: (T(X \otimes Y), \mu_{X \otimes Y}) \to (T(X),
\mu_{X}) \otimes (T(Y), \mu_{Y}),\end{equation}
for all $X, Y \in \ca$, and  $\ele_0 = \phi: (T(\uno), \mu_\uno) \to (\uno, \phi)$.
Moreover, the pair $(\ele: \ca \to \ca^T, \Fc: \ca^T \to \ca)$ is a
\emph{comonoidal adjunction} in the sense of \cite[Subsection 2.5]{blv}; see \cite[Example 2.4]{blv}.

\medbreak If $(T, \xi,  \phi)$ is   bimonad on the monoidal category  $\ca$,
then $(\tb, \xib, \overline \phi)$ is
a  bimonad on $\ca^{\op}$, where  $\tb=T$  as monads, $\overline \phi = \phi$,
and  $ \xib_{X, Y}=\xi_{Y,X}$ for all
$X, Y\in \ca$.

\begin{lema}\label{op-equiv}
 The identity functor induces a strict equivalence of monoidal
categories  $(\ca^{\op})^{\tb}\simeq (\ca^T)^{\op}$.\qed
\end{lema}
\medbreak
 Let $\ca$ be a monoidal category. A bimonad $T$ on $\ca$ is a \emph{Hopf monad}
if the \emph{fusion operators} $H^l$ and $H^r$ defined, for all $X, Y \in \ca$,
by
\begin{align} & H^l_{X, Y}: =  (\id_{T(X)} \otimes \mu_Y) \, \xi_{X, T(Y)}: T(X
\otimes T(Y)) \to T(X) \otimes T(Y),\\
& H^r_{X, Y}: = (\mu_X \otimes \id_{T(Y)}) \, \xi_{T(X), Y}: T(T(X) \otimes Y)
\to T(X) \otimes T(Y),\end{align}
are isomorphisms \cite[Subsection 2.7]{blv}.
If $\ca$ is a rigid monoidal category and $T$ is a  Hopf monad on $\ca$ then the
category $\ca^T$ is rigid  \cite[Subsection 3.4]{blv}.

\medbreak \begin{rmk}\label{hopf-adj}
Suppose that $T$ is a Hopf monad on $\ca$. It follows from \cite[Theorem
2.15]{blv} that $(\ele, \Fc)$ is indeed a \emph{Hopf adjunction}, that is, the
left and right Hopf operators $\mathbb H^l$ and $\mathbb H^r$ defined, for every
$Y \in \ca$, $(M, r) \in \ca^T$, by
\begin{align}& \mathbb H^l = (\id_{\ele(Y)} \otimes r) \, \xi_{Y, M}: \ele(Y
\otimes M) \to \ele(Y) \otimes \ele(M),\\
& \mathbb H^r = (r \otimes \id_{\ele(Y)}) \, \xi_{M, Y}: \ele(M \otimes Y) \to
\ele(M) \otimes \ele(Y),
\end{align}
are isomorphisms. \end{rmk}

 \begin{rmk}\label{finite-tensor-ct} Suppose  that $\ca$ is a $\ku$-linear
Abelian category
 and $T$ is a $\ku$-linear right exact monad on $\ca$.
Then the category $\ca^T$ is $\ku$-linear Abelian and the forgetful functor $\Fc
: \ca^T \to \ca$ is $\ku$-linear exact.
 In this case, if  $\ca$ is a tensor category over $\ku$,
then $\ca^T$ is a tensor category over $\ku$  and   the forgetful functor $ \Fc
: \ca^T \to \ca$
is a tensor functor
\cite[Proposition 2.3]{BN}.

 \begin{lema}\label{cat-finite} Suppose that $\ca$ is a finite tensor category.   Then so is $\ca^T$.\end{lema}

\begin{proof} The assumption that the tensor category $\ca$ is finite is equivalent to the assumption it has a projective generator $P$, that is, an object $P$ of $\ca$ such that the functor $\Hom_\ca(P, -)$ is faithful exact.
Let  $\ele: \ca \to \ca^T$ be the left adjoint of the forgetful functor $\Fc: \ca^T \to \ca$. By adjointness, we obtain a natural isomorphism
$\Hom_{\ca^T}(\ele(P), -) \cong  \Hom_\ca(P, -) \circ \Fc$.     Since $\Fc$ is faithful and exact,
then $\Hom_{\ca^T}(\ele(P), -)$ is faithful exact, that is, $\ele(P)$ is a projective generator of $\ca^T$.  Thus $\ca^T$ is a finite
tensor category as claimed.
\end{proof}
\end{rmk}

\begin{exa}\label{hm-equiv}   Let $G$ be a finite
group and let $\rho: \underline{G}\to \underline{\text{Aut}}_\otimes(\ca)$
 be an action of $G$ on $\ca$ by tensor autoequivalences.
In  other  words, for any $g\in G$ we have  a tensor functor
 $(\rho^g, \zeta_g):\ca\to\ca$, and for any $g,h\in G$,
 there are natural  isomorphisms of tensor functors
$\gamma_{g, h}:\rho^g\circ \rho^h\to \rho^{gh}$ and $\rho_0: \id_\ca \to \rho^e$. Associated to such an action there is a tensor category $\ca^G$, called the \emph{equivariantization} of $\ca$ under the action $\rho$, endowed with a canonical dominant tensor functor $\ca^G \to \ca$.

It was shown in \cite[Theorem 5.21]{BN} that the action $\rho$ induces a Hopf monad $T^\rho$ on $\ca$ in such a way that $\ca^{T^\rho} \cong \ca^G$ as tensor categories over $\ca$. The Hopf monad $T^\rho$ is defined in the form $T^\rho(X) = \bigoplus_{g \in G}\rho^g(X)$, with 
multiplication $\mu : (T^\rho)^2 =\bigoplus_{g,h} \rho^g\rho^h \to T^\rho=\bigoplus_{t \in G} \rho^t$ and unit $\eta: \id_\ca \to T^\rho=\bigoplus_{g\in G}\rho^g$, defined componentwise by the morphisms $\gamma_{g, h} : \rho^g\rho^h \to \rho^{gh}$, and by  $\rho_0: \id_\ca \to \rho^e$, respectively.   The comonoidal structure morphisms $\xi_{X,Y}: \bigoplus_{g \in G} \rho^g(X \otimes Y) \to \bigoplus_{s,t \in G} \rho^s(X) \otimes \rho^t(Y)$,  and $\phi:  \bigoplus_{g \in G} \rho^g(\uno)\to \uno$, are defined componentwise by the strong comonoidal structure of the tensor functors $\rho^g$.

The Hopf monad $T^\rho$ is moreover \emph{normal} in the sense that $T^\rho(\uno)$ is a trivial object of $\ca$; see Section \ref{mc-exact}.
\end{exa}

\subsection{The category  $\caT$}  Let $\ca$ be a finite tensor category over
$\ku$ and let  $T$ be a $\ku$-linear right exact Hopf monad on $\ca$ . The
category
$\ca$ is a $\ca^T$-module through the tensor functor $\Fc: \ca^T\to \ca$.
This means that the action $\otb: \ca^T\times \ca\to \ca$ is given by
$(X, s)\otb Y= X\ot Y$, for all $(X, s)\in \ca^T$, $Y\in  \ca$.

\medbreak 

Suppose that  $T$ is a \emph{faithful} Hopf
monad, or equivalently, that $\Fc$ is a dominant tensor functor
\cite[Proposition 4.1]{BN}. Then $\ca$
is an exact indecomposable $\ca^T$-module; see Example \ref{dominant-exact}.

\medbreak

We shall use the notation $\caT$ to indicate the category $\End_{\ca^T}(\ca)$ of
 $\ku$-linear \esn module endofunctors of $\ca$.

 Observe that since $\ca^T$ is a finite tensor category (Lemma
\ref{cat-finite}) and $\ca$ is an exact $\ca^T$-module, then every
$\ku$-linear module endofunctor of $\ca$ is exact \cite[Proposition 3.11]{eo}.
Moreover, it follows from \cite[Proposition 3.23]{eo} that the category $\caT$
is again a finite tensor category over $\ku$.  

\medbreak By \cite[Lemma 3.25]{eo},  $\ca$ is
an exact indecomposable  $\caT$-module with respect to the action
$\otb:(\caT)\times \ca\to \ca$ given by
$$F\otb X= F(X),$$
for all $F\in \caT, X\in \ca$.

The third part of the next lemma is a particular case of \cite[ Theorem 
3.27]{eo}.
 We shall  include  the proof for the reader's convenience.

\begin{lema}\label{lema-tec1}  Let $\ca$ be a finite tensor category and
let $T$ be a right exact faithful Hopf monad on $\ca$. Then the following hold:
\begin{itemize}
      \item[1.] The functor $R:\ca^{\op}\to \caT$, $R(X)=R_X$, where for any
$X\in \ca$,
$R_X:\ca\to \ca$ is given by $R_X(Y)=Y\ot X$, is  a full embedding of tensor
categories.

\item[2.] Let, for all $(X, s) \in \ca^T$, $Y \in \ca$,
$$b_{X,Y}: T(X\ot Y) \xrightarrow{\;\;\xi_{X,Y} \;\;} T(X)\ot T(Y)
\xrightarrow{\;\;s\ot\id \;\;} X\ot T(Y).$$
Then $(T, b)$ is an algebra in $\caT$, with multiplication $\mu: T^2 \to T$ and
unit $\eta: \id_\ca \to T$.

\item[3.]  There is  an equivalence of tensor categories
 $\ca^T\simeq \End_{\caT}(\ca)$.
             \end{itemize}

\end{lema}
\pf
1. It is not difficult to see that for any $X\in \ca$ the functor $R_X:\ca\to
\ca$ is a
$\ca^T$-module functor and $R_{X\ot Y}= R_Y\circ R_X$.

\medbreak

2.   It follows from  Remark \ref{hopf-adj} that the composition
$$b_{X,Y}: T(X\ot Y) \xrightarrow{\;\;\xi_{X,Y} \;\;} T(X)\ot T(Y)
\xrightarrow{\;\;s\ot\id \;\;} X\ot T(Y)$$
is an isomorphism for any $(X,s)\in \ca^T$, $Y \in \ca$.

To show that $(T, b)$ is a module functor we have to prove that equations
 \eqref{modfunctor1}, \eqref{modfunctor2} are fulfilled.
Let $(X,s), (Y,r)\in \ca^T$ and $Z\in \ca$. In this case the left hand side of
\eqref{modfunctor1} equals
\begin{align*}(\id_X  \otb  b_{Y,Z})b_{X,Y\ot Z}&=
(\id_X\ot (r\ot \id_{T(Z)}) \xi_{Y,Z}) (s\ot \id_{T(Y\ot Z)}) \xi_{X,Y\ot Z}\\
&=(s\ot r\ot \id_{T(Z)}) (\id_X\ot \xi_{Y,Z}) \xi_{X,Y\ot Z}\\
&= (s\ot r\ot \id_{T(Z)})(\xi_{X,Y}\ot\id_{T(Z)}) \xi_{X\ot Y, Z}\\
&= b_{X\ot Y, Z}.
\end{align*}
The third equality follows from \eqref{hopfmonad1}. This proves equation
\eqref{modfunctor1}.
Equation \eqref{modfunctor2} follows  similarly, using \eqref{hopfmonad2}.
Hence $(T, b) \in \caT$.

In the same fashion, now using the relations
\eqref{module:monads}, \eqref{hopfmonad11} and the naturality of $\xi$,
 it is shown that $\mu: (T, b) \otimes (T, b) \to (T, b)$ and $\eta: (\id_\ca,
\id) \to (T, b)$
satisfy condition \eqref{modfunctor3}, that is, they are morphisms of
 $\ca^T$-module functors. This implies that $(T, b)$ is an algebra in $\caT$ as
claimed.

\medbreak
3.  Define $\Phi: \ca^T \to \End_{\caT}(\ca)$ by $\Phi(X,s)(Y)=X\ot Y$ for all
$(X,s)\in \ca^T, Y\in\ca$. The functor $\Phi(X,s)$ is indeed a $\caT$-module
functor
as follows. If $(F, d^F)\in \caT$, $Y\in \ca$ define
$$c^{(X,s)}_{F, Y}: \Phi(X,s)(F\otb Y)\to F\otb \Phi(X,s)(Y), \quad
c^{(X,s)}_{F, Y}=
\big(d^F_{X,Y}\big)^{-1}.$$
Equation \eqref{modfunctor1} amounts to
$$F(c^{(X,s)}_{G, Y}) c^{(X,s)}_{F, G(Y)}= c^{(X,s)}_{F\circ G, Y}.$$
This last equation is equivalent to equation \eqref{modfunctor-comp}.
Define $\Psi:\End_{\caT}(\ca)\to \ca^T$ by $\Psi(F, c^F)=(F(\uno), s^F)$,
where
$$s^F= F(\phi)\circ \big(c^F_{T,\uno}\big)^{-1}.$$
 It  is straightforward to see that
that $\Phi$ and $\Psi$ are tensor functors and
 thus  they give  the desired  equivalence of tensor categories.
\epf

\begin{exa}\label{bimodulos} Let $\ca$ be a finite tensor category and
let $T$ be a right exact faithful Hopf monad on $\ca$. Suppose that the
forgetful functor $\Fc: \ca^T \to \ca$  is perfect, that is, it has an exact
right adjoint. Let $(A, \sigma) \in \mathcal Z(\ca^T)$ be the induced central
algebra of  $T$. 

As explained in Subsection \ref{tensor-functor}, the category
$(\ca^T)_A$ of right $A$-modules in $\ca^T$ is a tensor category and there is an
equivalence of tensor categories $K: (\ca^T)_A \to \ca$ such that $KF_A = \Fc$,
where $F_A: \ca^T \to (\ca^T)_A$ is the free module functor.

 Let ${}_A(\ca^T)_A$ be the category  of
$A$-bimodules in $\ca^T$; ${}_A(\ca^T)_A$ is a tensor category with tensor
product $\otimes_A$ and unit object $A$.  Observe that the functor $K$ induces
an equivalence of $\ca^T$-module categories $(\ca^T)_A \simeq
\ca$. We thus obtain an equivalence of tensor categories $\caT \simeq
({}_A(\ca^T)_A)^{\op}$; see \cite[Proof of Lemma 3.25]{eo}.

Under this equivalence, the full embedding in Lemma \ref{lema-tec1}  (1)
corresponds to the full embedding $(\ca^T)_A \to {}_A(\ca^T)_A$ induced by the
half-braiding $\sigma$.
\end{exa}

\section{Module categories over $\ca^T$}\label{mc-cat}

Along this section $\ca$ will be a monoidal category. For any Hopf monad $T$ on $\ca$
we shall give a construction of module categories over $\ca^T$ from
module categories over $\ca$. 
The monad structure is denoted by $(T,\mu,\eta)$
and the comonoidal structure by $(T,\xi,\phi)$.

\subsection{$T$-equivariant module categories}\label{subsection:t-eq}

Let $T$ be a Hopf monad over  the monoidal category $\ca$
 and let $\Mo$ be a $\ca$-module.

The functor $T$, being  comonoidal, induces a structure of lax $\ca$-module
category on $\Mo$, denoted $\Mo(T)$. The action on $\Mo(T)$ is given in the form
$X\otb^T M = T(X) \otb M$, and the  associativity and unit morphisms are defined
by $$m_{X,Y,M}^T = m_{T(X),T(Y), M} \, (\xi_{X,Y}\otb\, \id_M), \quad l^T_M =
l_M (\phi \otb \id_M),$$
for all $X, Y\in \ca$, $M\in \Mo$.
See Subsection \ref{h-monads}.

\begin{lema}\label{structure of square}  Let  $(U,c):\Mo\to \Mo(T)$ be a
lax $\ca$-module functor. Then the pair $(U^2, d):\Mo\to \Mo(T)$ is a lax
$\ca$-module functor, where
\begin{equation}\label{d:map} d_{X,M}=(\mu_X\otb
\id_{U^2(M)})c_{T(X),U(M)}U(c_{X,M}),
\end{equation}
for all $X\in\ca$, $M\in\Mo$.
\end{lema}
\pf
 For every $M\in \Mo$ we have
\begin{align*} l^T_{U^2(M)} d_{\bf 1, M} & = l_{U^2(M)} (\phi \otimes
\id_{U(M)}) d_{\bf 1, M} \esn \\
&  l_{U^2(M)} (\phi\mu_\uno\otb \id_{U^2(M)}) c_{T(\uno), U(M)} U(c_{\uno, M})\\
&= l_{U^2(M)} (\phi T(\phi) \otb \id_{U^2(M)}) c_{T(\uno), U(M)} U(c_{\uno,
M})\\
&= l_{U^2(M)} (\phi  \otb \id_{U^2(M)}) c_{\uno, U(M)} U(\phi\otb \id_{ U(M) })
U(c_{\uno, M})\\
&= U(l_{U(M)}) U( (\phi\otb \id_{ U(M) })c_{\uno, M} )\\
&= U^2(l_M),
\end{align*}
 the second equality by \eqref{hopfmonad11}, the third equality by the
naturality of $c$, and the fifth because $c$ is a lax module functor. Hence $d$
satisfies \eqref{modfunctor2}. \esn
 Now  the left hand side of equation
 \eqref{modfunctor1} equals
\begin{multline*} (\mu_X\otb  ( \mu_Y \otb \id_{U^2(M)})) (\id_{ T^2(X)} \otb
c_{T(Y), U(M)}U(c_{Y,M}))\\
\times c_{T(X),U(T(Y)\otb M))}U(c_{X,Y \otb M})U^2(c_{m_{X,Y,M}}).
\end{multline*}
 On the other hand,  the right hand side of
 \eqref{modfunctor1} equals
\begin{align*} & m_{T(X), T(Y),U^2(M)} (\xi_{X,Y}\mu_{X\ot Y}\otb\id_{U^2(M)})
c_{T(X\ot Y), U(M)} U(c_{X\ot Y,M})\\
& =m_{T(X),T(Y),U^2(M)} ((\mu_X\ot
\mu_Y)\xi_{T(X),T(Y)}T(\xi_{X,Y})\otb\id_{U^2(M)}) \\
& \qquad  \times c_{T(X\ot Y), U(M)} U(c_{X\ot Y,M}) \\
& = m_{T(X),T(Y),U^2(M)} ((\mu_X\ot \mu_Y)\xi_{T(X),T(Y)}\otb\id_{U^2(M)})
\\
& \qquad   \times  c_{T(X)\ot T(Y),M}U(\xi_{X,Y}\ot\id_M)U(c_{X\ot Y,M}) \\
& =(\mu_X\otb (\mu_Y\otb\id_{U^2(M)})) m_{T^2(X),T^2(Y),U^2(M)}
(\xi_{T(X),T(Y)}\otb\id_{U^2(M)})
\\
& \qquad   \times  c_{T(X)\ot T(Y),M}U(\xi_{X,Y}\otb\id_M)U(c_{X\ot Y,M}) \\
& =(\mu_X\otb (\mu_Y\otb\id_{U^2(M)})) (\id_{T^2(X)}\otb c_{T(Y), U(M)})c_{T(X),
T(Y)\otb U(M)}
\\
& \qquad   \times  U(m_{T(X),T(Y),U(M)}(\xi_{X,Y}\otb\id_M)c_{X\ot Y,M}) \\
& =(\mu_X\otb( \mu_Y\otb\id_{U^2(M)})) (\id_{T^2(X)}\otb c_{T(Y), U(M)})c_{T(X),
T(Y)\otb U(M)}
\\
& \qquad   \times  U((\id_X\otb c_{Y,M})c_{X,Y\otb M})U^2(c_{m_{X,Y,M}})
\end{align*}
The second equality follows from \eqref{hopfmonad3}, the third equality follows
from the
naturality of $c$,  the fourth by the naturality of $m$, the fifth and sixth
equalities follow
 because $(U,c)$
is a lax module functor. \epf

\begin{defi}\label{defi-temc} A \emph{$T$-equivariant $\ca$-module category} is a triple $(\Mo,
U, c)$, where:
\begin{itemize}\item $\Mo$ is a $\ca$-module category,
\item $(U, c): \Mo \to \Mo(T)$ is a lax $\ca$-module functor,
\item $(U,\nu,u)$ is a monad on $\Mo$ and both natural transformations
$\nu: U^2 \to U$ and $u: \Id_{ \Mo } \to U$ are morphisms of $\ca$-module
functors.
\end{itemize}

\end{defi}
\begin{rmk}  The module functor structures of $\Id_{\Mo}:\Mo\to \Mo(T)$ and
$U^2: \Mo \to \Mo(T)$, implicit in the third
condition of the definition, are the ones given by $\eta \otb \id$ and Lemma
\ref{structure of square}, respectively.
\end{rmk}

\medbreak Let
 $(\Mo, U, c)$  be a $T$-equivariant $\ca$-module.
 Since $U$ is a monad on $\Mo$, we may consider
the category $\Mo^U$ of $U$-modules in $\Mo$.
The objects $(M, s) \in \Mo^U$ will be called $U$-equivariant objects.

\begin{exa}  The monoidal category $\ca$   is a module category over itself and
 $(T,\xi):\ca\to \ca(T)$ is a lax $\ca$-module functor. It follows from
 \eqref{hopfmonad3} that $\mu:(T^2,d)\to (T,\xi)$ is a natural transformation of
module functors, hence $\ca$ is a $T$-equivariant $\ca$-module category.
$T$-modules $\ca^T$.
\end{exa}

\begin{exa}  (Module categories over equivariantizations.)  Let $G$ be a finite
group and $\rho: \underline{G}\to \underline{\text{Aut}}_\otimes(\ca)$
 be an action of $G$ on $\ca$.
As explained in Example \ref{hm-equiv}, the endofunctor $T^\rho= \oplus_{\sigma\in G} \,
\rho^\sigma$
 has a structure of Hopf monad over $\ca$ such that the equivariantization
$\ca^G$ is tensor equivalent to $\ca^{T^\rho}$.

Let $F\subseteq G$ be a subgroup. Recall that  an $F$-equivariant
$\ca$-module \cite{ENO2}  is a module category $\Mo$ over
$\ca$ endowed with a family of module functors
$(U_\sigma, c^\sigma):\Mo\to \Mo(\rho^\sigma)$ for any $\sigma\in F$
 and a family of natural isomorphisms $\mu_{\sigma,\tau}:(U_\sigma\circ U_\tau,
b)\to (U_{\sigma\tau}, c^{\sigma \tau})$ $\sigma, \tau\in F$ such that
\begin{equation}\label{mod-equi1} {(\mu_{\sigma,\tau \nu})}_M\circ U_\sigma(
\mu_{\tau,\nu})_M={(\mu_{\sigma \tau,\nu})}_M\circ
{(\mu_{\sigma,\tau})}_{U_\nu(M)},
\end{equation}
\begin{equation}\label{mod-equi2}
c^{\sigma\tau}_{X,M}\circ(\mu_{\sigma,\tau})_{X\otb M}=
((\gamma_{\sigma,\tau})_X\otb (\mu_{\sigma,\tau})_M)\circ
c^\sigma_{\rho^\tau(X),U_\tau(M)}\circ
U_\sigma(c^\tau_{X,M}),
\end{equation}
for all $\sigma,\tau,\nu\in F$, $X\in \ca$, $M\in \Mo$. The functor $U:\Mo\to
\Mo(T^\rho)$,
$U=\oplus_{\sigma\in F} U_\sigma$ makes the category $\Mo$ $T^\rho$-equivariant.
The category of $U$-equivariant objects in $ \Mo$ coincides with the category of
$F$-equivariant objects in $\Mo$ in the sense of \cite{ENO2}.
\end{exa}

\begin{exa}\label{exa:mc-hopf}  (Module categories over $H$-mod.)   Consider  the tensor category 
$\ca = \vect_\ku$ of finite dimensional $\ku$-vector spaces. Let $H$ be a finite
dimensional Hopf algebra over $\ku$. Then $H \otimes -: \ca \to \ca$ is a
$\ku$-linear right exact Hopf monad on $\ca$ and there is an equivalence of
tensor categories $\ca^{H \otimes -} \simeq H$-mod, where $H$-mod denotes the
category of finite dimensional representations of $H$. Moreover, this assignment
defines an equivalence between finite dimensional Hopf algebras over $\ku$ and
$\ku$-linear Hopf monads on $\ca$ \cite[Lemma 2.5]{BN}.

Let $\Mo = \ca = \vect_\ku$ denote the canonical $\ca$-module category (which
is, up to equivalence, the only indecomposable $\ca$-module category). Let also 
$H$ be a finite dimensional Hopf algebra over $\ku$ and $T = H \otimes -$ the
associated Hopf monad. Then we have:

\begin{lema} There is an equivalence between the categories of
$T$-equivariant $\ca$-module category structures on $\Mo$ and finite
dimensional left $H$-comodule algebras. \end{lema}

\begin{proof} Every finite dimensional $\ku$-algebra $A$ induces canonically a
$\ku$-linear monad $U = A \otimes_\ku -$ on $\vect_\ku$. Conversely, every such
monad
$U$ is isomorphic to the one induced by a finite dimensional $\ku$-algebra $A$:
indeed, as in the proof of \cite[Lemma 2.5]{BN}, we get that $U \simeq U(\bf 1)
\otimes -$
as $\ku$-linear functors, and since  $U$ is a monad then $A = U(\bf 1)$ is an
algebra \cite[Example 1.2]{BV}.

Under this correspondence, the conditions on $U = A \otimes -$ in
Definition \ref{defi-temc}
correspond to the condition that $A$ is a left $H$-comodule algebra. Indeed, a
structure $c$ of lax $\ca$-module functor on $U$ is uniquely determined by a map
$c: A \to A \otimes H$ making $A$ into a right $H$-comodule, in view of
\eqref{modfunctor1} and \eqref{modfunctor2}. The requirement that the
multiplication and unit of $U = A \otimes -$ are morphisms of module functors
amounts to the condition that the multiplication and unit of $A$ are comodule maps. 
Thus $A$ is a right $H$-comodule algebra, as claimed.
\end{proof}
\end{exa}

 \begin{defi}  Let $(\Mo, U, c)$, $(\Mo, \widetilde{U}, \widetilde{c})$ be two
$T$-equivariant
structures on the module category $\Mo$.

\begin{itemize}
 \item  A \emph{morphism of $T$-equivariant structures}
 $\alpha: (\Mo, U, c)\to
(\Mo, \widetilde{U}, \widetilde{c})$ is a monad morphism $\alpha:U\to
\widetilde{U}$
such that $\alpha$ is also a morphism of lax module functors.

\item A morphism of $T$-equivariant structures
 $\alpha: (\Mo, U, c)\to
(\Mo, \widetilde{U}, \widetilde{c})$ is \emph{surjective} if $\alpha_M$ is
surjective
for any $M\in \Mo$.
\item  We say that a
$T$-equivariant
structure $(\Mo, U, c)$ is \emph{simple} if any surjective
morphism of $T$-equivariant structures $\alpha:U\to \widetilde{U}$ is an
isomorphism.
\end{itemize}

\end{defi}

\begin{teo}\label{U-equivariant modcat} Let $\Mo$ be a $T$-equivariant
$\ca$-module
with lax module functor $(U,c):\Mo\to \Mo(T)$. Then the following holds.
\begin{enumerate}
 \item[1.] The category $\Mo^U$ has a structure of
 $\ca^T$-module category.

\item[2.]  If $\alpha: (\Mo, U, c)\to
(\Mo, \widetilde{U}, \widetilde{c})$ is a morphism of $T$-equivariant
structures, the functor
$\alpha^*:\Mo^{\widetilde{U}}\to \Mo^U $, $\alpha^*(M,s)=(M, s\circ\alpha_M)$
for all $
(M,s)\in \Mo^{\widetilde{U}}$, is a $\ca^T$-module functor.

\item[3.]   If  $\Mo^U$
 is a simple  $\ca^T$-module category then $(\Mo, U, c)$ is simple.
\end{enumerate}
\end{teo}
\pf 1. Let us define the action $\otb:\ca^T\times \Mo^U\to \Mo^U$ as follows.
Let $(X,r)$ be an object in  $\ca^T$ and $(M,s)$ be a $U$-equivariant object in
 $\Mo$, then $(X,r)\otb (M,s)=(X\otb M, (r\otb s) c_{X,M})$.  Let us
prove that the object $(X\otb M, (r\otb s) c_{X,M})$ is $U$-equivariant. For
this,
we need to show that the map $t=(r\otb s) c_{X,M}$
 satisfies $t\circ U(t)=t\circ \nu_{X\otb M}$. Indeed:
\begin{align*}t\circ U(t)&=(r\otb s) c_{X,M} U(r\otb s)U(c_{X,M})
\\ &=(rT(r)\otb sU(s))c_{T(X),U(M)} U(c_{X,M})\\
&=(r \mu_X\otb s\nu_M)c_{T(X),U(M)} U(c_{X,M})\\
&=(r \otb s \nu_M) d_{X,M}= (r \otb s) c_{X,M} \nu_{X\otb M}=t\circ \nu_{X\otb
M}.
\end{align*}
The first equality follows from the naturality of $c$,
the second follows because $(X,r)$ is a $T$-module
and $(M,s)$ is $U$-equivariant. The third equality follows
 from the definition of $d$, see equation \eqref{d:map},
 the last equality follows since $\nu$ is a module functor.
The associativity  and unit isomorphisms are the obvious ones.

\medbreak

2.  Since $\alpha: U \to \widetilde U$ is a morphism of monads, it follows from
 \cite[Lemma 1.7]{BV} that $\alpha$ induces a functor
$\alpha^*: \Mo^{\widetilde{U}} \to \Mo^U$ such that $\Fc_U \alpha^* =
\Fc_{\widetilde U}$.

Let us show that $ \alpha^*$ is a $\ca^T$-module functor. Let be $(X, r)\in
\ca^T$, $(M,s)\in \Mo^U$. Then
$$\alpha^*((X,r)\otb (M,s))= (X\otb M, (r\otb s) c_{X,M}\alpha_{X\otb M}),$$
and
$$(X,r)\otb \alpha^*(M,s)=(X\otb M, (r\otb s\alpha_M) \widetilde{c}_{X,M}).$$
Since $\alpha: U \to \widetilde U$ is a morphism of lax module functors, it
follows from
equation \eqref{modfunctor3} that $\alpha^*((X,r)\otb (M,s))=(X,r)\otb
\alpha^*(M,s)$.
\medbreak

 3.  Assume $(\Mo, \widetilde{U}, \widetilde{c})$ is another
$T$-equivariant structure and $\alpha:U\to \widetilde{U}$ is a surjective
morphism of
$T$-equivariant structures.
Notice that $\alpha^*(f) = f$ for all morphism $f$ in $\Mo^{\widetilde{U}}$.
Hence $\alpha^*$ is a faithful functor.

The functor $\alpha^*$ is also full. Indeed, let be $(M,s),
(N,r)\in\Mo^{\widetilde{U}}$
and $f: (M,s\alpha_M) \to (N,r\alpha_N)$ be a morphism in $\Mo^U $. Then
$f s\alpha_M= r\alpha_N U(f)$, which implies, by
the naturality of $\alpha$, that $f s\alpha_M= r \widetilde{U}(f)\alpha_M$.
Since
$\alpha_M$ is surjective, then $f s= r \widetilde{U}(f)$ and $f$
is a morphism in $\Mo^{\widetilde{U}}$. \medbreak

Since $\Mo^U$
 is a simple  $\ca^T$-module category then the functor $\alpha^*$
is an equivalence of module categories. Hence there exists a module functor
$\Fc: \Mo^U \to \Mo^{\widetilde{U}}$ such that $\Fc\circ \alpha^*\simeq \Id$
and $\alpha^*\circ \Fc\simeq \Id$. In particular, there exists a natural
isomorphism
$\gamma: \alpha^*\circ \Fc \to \Id$. Since $\gamma_{(M,s)}$ is a morphism in the
category $ \Mo^U $, for all
$(M,s)\in  \Mo^U $, the diagram
\begin{equation}\label{diagram11}
\xymatrix{
U(\Fc(M)) \ar[d]^{s^\Fc\alpha_{\Fc(M)}}\ar[rr]^{U(\gamma)}&& U(M) \ar[d]_{s} \\
\Fc(M)\ar[rr]^{\gamma}&& M}
\end{equation}
is commutative. Here $\Fc(M,s)=(\Fc(M), s^\Fc)$. Let us define a new functor
$\widehat{\Fc}:\Mo^U \to \Mo^{\widetilde{U}}$ as follows. For any $(M,s)\in
\Mo^U$
$$\widehat{\Fc}(M,s)=(M,\gamma_M\circ s^\Fc \circ\widetilde{U}(\gamma^{-1}_M)).
$$
For any morphism $f:(M,s)\to (N,r)$ in $\Mo^U$, $\widehat{\Fc}(f)=f$.

The object $\widehat{\Fc}(M,s)$ is  in $\Mo^{\widetilde{U}}$. Indeed,
denote $t=\gamma_M\circ s^\Fc \circ\widetilde{U}(\gamma^{-1}_M)$, then
\begin{align*} t\widetilde{U}(t)&= \gamma_Ms^\Fc\widetilde{U}(\gamma^{-1}_M)
\widetilde{U}( \gamma_Ms^\Fc \widetilde{U}(\gamma^{-1}_M) )=
\gamma_Ms^\Fc \widetilde{U}(s^\Fc)\widetilde{U}^2(\gamma^{-1}_M)
\\
&=\gamma_M s^\Fc \widetilde{\nu}_{\Fc(M)}\widetilde{U}^2(\gamma^{-1}_M)
=\gamma_Ms^\Fc   \widetilde{U}(\gamma^{-1}_M) \widetilde{\nu}_M=t 
\widetilde{\nu}_M.
\end{align*}
The third equality follows since $(\Fc(M), s^\Fc)\in \Mo^{\widetilde{U}}$ and
the fourth equality follows from the naturality of $\widetilde{\nu}$. Now,
using \cite[Lemma 1.6]{BV}, it follows that there exists a monad morphism
$\beta: U\to \widetilde{U}$ such that $\widehat{\Fc}=\beta^*$.
Whence, for any $(M,s) \in \Mo^U$ we have $\gamma_Ms^\Fc= s \beta_M
\widetilde{U}(\gamma_M)$. Hence
\begin{align*}\gamma_Ms^\Fc \alpha_{\Fc(M)}&= s \beta_M
\widetilde{U}(\gamma_M)\alpha_{\Fc(M)}\\
&=s \beta_M  \alpha_M U(\gamma_M).
\end{align*}
Using commutativity of diagram \eqref{diagram11}, we obtain that $s \beta_M 
\alpha_M= s$. Since this argument
can be applied to $(U(M), \nu_M)$ for any $M\in \Mo$, we get that
$\nu_M\beta_{U(M)}
  \alpha_{U(M)}=\nu_M$.
Thus
\begin{align*}\id_M=\nu_M U(u_M)=\nu_M\beta_{U(M)}  \alpha_{U(M)} U(u_M)
=\beta_M  \alpha_M.
\end{align*}
An analogous argument shows that $\alpha_M \beta_M=\id_M$, and therefore
$\alpha$ is an isomorphism. Thus we conclude that the equivariant structure
$(\Mo, U, c)$ is simple.
\epf

Let $\Mo, \widetilde{\Mo}$ be $\ca$-module categories and assume that
$(\Mo, U, c)$, $(\widetilde{\Mo}, \widetilde{U}, \widetilde{c})$ are
$T$-equivariant
structures. Let $(G, d):\Mo\to \widetilde{\Mo}$ be a $\ca$-module functor.
We shall denote by $\widehat{T}(G): \Mo(T)\to \widetilde{\Mo}(T)$
the following $\ca$-module functor. For any
$M\in   \Mo(T)$, $\widehat{T}(G)(M)=G(M)$, and the module structure is given by
$$d_{T(X), M}: G(T(X) \otb M) \to T(X) \otb G(M).$$ Assume there is a natural
transformation
$\theta: \widetilde{U}\circ G\to \widehat{T}(G)\circ U$. Define the functor
$\widehat{G}: \Mo^U\to
\widetilde{\Mo}^{\widetilde{U}}$ by $\widehat{G}(M,s)=(G(M), G(s)\theta_M)$
for all $(M,s)\in  \Mo^U$.
\begin{prop} Using the above notation, the following assertions hold.
\begin{itemize}
 \item[1.] Suppose there exists a $\ca$-module natural transformation
$ \theta: \widetilde{U}\circ G\to \widehat{T}(G)\circ U$ such that
\begin{equation}\label{2-nat-f-eq}
\theta_M \widetilde{\nu}_{G(M)}=G(\nu_M)\theta_{U(M)}\widetilde{U}(
\theta_M),\quad
 \theta_M \widetilde{u}_{ G(M)}=G(u_M),
\end{equation}
for all $M\in \Mo$. Then the functor $\widehat{G}$ is a
$\ca^T$-module functor.
\item[2.]  Assume that $(H,h):\Mo\to \widetilde{\Mo}$ is another $\ca$-module
functor
equipped with a $\ca$-module natural transformation
$\chi: \widetilde{U}\circ H\to \widehat{T}(H)\circ U$ satisfying
\eqref{2-nat-f-eq}
(replacing $G$ by $H$)
and let
$\alpha:(G, d)\to (H,h)$ be a $\ca$-module natural transformation such that
\begin{equation}\label{2-nat-f-eq2}
 \alpha_{U(M)}\theta_M=\chi_M  \widetilde{U}(\alpha_M),
\end{equation}
for all $M \in \Mo$. Then, the natural transformation
$\widehat{ \alpha}:\widehat{G}\to \widehat{H}$, $\widehat{
\alpha}_{(M,s)}=\alpha_M$, $(M,s) \in  \Mo^U$,  is a $\ca^T$-module natural
transformation.
\end{itemize}
\end{prop}
\pf 1. That the functor  $\widehat{G}$ is well-defined, that is,
$\widehat{G}(M,s)\in \widetilde{\Mo}^{\widetilde{U}}$ for any
 $(M,s)\in  \Mo^U$ is a consequence of \eqref{2-nat-f-eq}.
 The module structure of the functor $\widehat{G}$ is $d$, the same module
structure
of the functor $G$. This map is a morphism in the category
$\widetilde{\Mo}^{\widetilde{U}}$
since $\theta$ is a $\ca$-module natural transformation.

2. For any $(M,s) \in  \Mo^U$ the map $\alpha_M$ is a morphism in the category
$ \widetilde{\Mo}^{\widetilde{U}}$ since it satisfies \eqref{2-nat-f-eq2}.
\epf

\subsection{Module categories over $\caT$}

Let us assume that $(T, \mu, \eta)$ is a Hopf monad over $\ca$. Then $\tb$ is a
Hopf
monad over $\ca^{\op}$. Let
$\No$ be a left $\catop$-module.
It follows from Lemma \ref{lema-tec1} (1) that $\No$ is a left $\ca$-module.
The left action is given by
$$\odot:\ca\times \No\to \No, \quad X\odot N= R_X\otb N,$$
for all $X\in \ca$, $N\in \No$, and the associativity
$$m_{X,Y,N}: (X\ot Y)\odot N\to X\odot (Y\odot N), m_{X,Y,N}=m_{R_X,R_Y,N},$$
 for all $X, Y\in \ca$, $N\in \No$. If $f:X\to Y$ is a morphism in $\ca$ and
$g:N\to M$
is a morphism in $\No$ then $f\odot g= \alpha_f\otb g$, where $\alpha_f:R_X\to
R_Y$ is the
natural transformation $(\alpha_f)_Z:X\ot Z\to Y\ot Z, (\alpha_f)_Z=f\ot \id_Z,$
for all
$Z\in \ca$.

Let $T$ be a  right exact faithful  Hopf monad on $\ca$.
By Lemma \ref{lema-tec1} (2), $(\tb, b)\in \catop$.
 Let $U:\No\to \No$ be
the functor defined by $U(N)=  \tb  \otb N$, for all $N\in \No$,  and let 
$\nu:U^2\to U$ and $u:\Id_{\No}\to U$ be the
natural transformations
$$\nu_N=(\mu\otb\id_N)m^{-1}_{T,T,N}, \quad u_N=\eta\otb \id_N,$$
for all $N\in \No$.
\begin{lema} $(U,\nu,u)$ is a monad  on $\No$.
\end{lema}
\pf  Since $\No$ is a left $\catop$-module, then there is a monoidal functor $L:
\catop\to \End(\No)$, defined in the form $L(X) = X \otb N$, $N \in \No$, where
the
monoidal structure $c_{X, Y}: L(X \otimes^{\op} Y) \to L(X) \otimes L(Y)$ is the
natural transformation given by
$$(c_{X, Y})_N = m_{X, Y, N}^{-1}, $$ for all $X, Y \in \catop$, $N \in \No$. In
particular, $L$ takes algebras in $\catop$ to algebras in $\End(\No)$, that is,
to monads on $\No$. This implies the lemma, in view of Lemma \ref{lema-tec1}
(2).
\epf

\begin{teo}\label{equi-mc}   Let $T$ be a right exact faithful Hopf monad on
$\ca$.  Then the following hold:
\begin{enumerate}
 \item[1.] With the above module structure, $\No$ is a $T$-equivariant
 left $\ca$-module.

\item[2.]  There is an equivalence of $\ca^T$-module categories
$$\No ^U\simeq \Fun_{\catop}(\ca^{\op}, \No).$$
\end{enumerate}

\end{teo}
\pf 1. We shall prove that the functor $U:\No\to \No$ gives a
$T$-equivariant structure on $\No$. 

Let us first show that $U:\No\to \No(T)$ is a lax module
functor.
For any $X\in \ca$, let $\xi^X: T\circ R_X\to R_{T(X)}\circ T$ be the natural
transformation given by $(\xi^X)_Y=\xi_{X,Y}$, $Y\in \ca$.
For any $X\in \ca, N\in\No$, define
$c_{X,N}:U(X\odot N)\to T(X)  \odot U(N)$ by
\begin{equation}\label{mod-2}
 c_{X,N}= m_{R_{T(X)},T,N} (\xi^X\otb \id_N) m^{-1}_{T,R_X,N}.
\end{equation}
Let us prove that for all $X, Y\in \ca$, $N\in \No$,
\begin{equation}\label{modf-tq}(\id_X\odot c_{Y,N}) c_{X,Y\odot N}
U(m_{R_X,R_Y,N})=m_{R_X,R_Y,T(N)}
(\xi_{X,Y}\odot \id_T(N))  c_{X\ot Y,N}.
\end{equation}

The right hand side of \eqref{modf-tq} equals
\begin{multline}\label{kk1}= m_{R_{T(X)},R_{T(Y)},T\otb N}
(\xi_{X,Y}\odot \id_{T\otb N}) m_{R_{T(X\ot Y)},T,N} (\xi^{X\ot Y}\otb \id_N)
m^{-1}_{T,R_{X\ot Y},N}\\
=m_{R_{T(X)},R_{T(Y)},T\otb N}m_{R_{T(X)\ot T(Y)},T,N}(\xi_{X,Y}\odot \id_N)
 (\xi^{X\ot Y}\otb \id_N) m^{-1}_{T,R_{X\ot Y},N}.
\end{multline}
The left hand side of \eqref{modf-tq} equals
\begin{multline*}= (\id_X\odot  m_{R_{T(Y)},T,N} (\xi^Y\otb \id_N)
m^{-1}_{T,R_Y,N})
m_{R_{T(X)},T,Y\odot N} (\xi^X\otb \id_{Y\odot N})\\
\times  m^{-1}_{T,R_X,{Y\odot N}} (\id_T\otb m_{R_X,R_Y,N})\\
=  (\id_X\odot  m_{R_{T(Y)},T,N} (\xi^Y\otb \id_N) m^{-1}_{T,R_Y,N})
m_{R_{T(X)},T,Y\odot N} (\xi^X\otb \id_{Y\odot N})\\
\times  m^{-1}_{T,R_X,{Y\odot N}}  m_{T,R_X,R_Y\otb N} m_{T\circ
R_X,R_Y,N}m^{-1}_{T,R_{X\ot Y},N} \\
= (\id_X\odot  m_{R_{T(Y)},T,N} (\xi^Y\otb \id_N) m^{-1}_{T,R_Y,N})
m_{R_{T(X)},T,Y\odot N} (\xi^X\otb \id_{Y\odot N})\\
\times m_{T\circ R_X,R_Y,N}m^{-1}_{T,R_{X\ot Y},N} \\
=(\id_X\odot  m_{R_{T(Y)},T,N} (\xi^Y\otb \id_N) m^{-1}_{T,R_Y,N})
m_{R_{T(X)},T,Y\odot N} m_{R_{T(X)}\circ T,R_Y, N}\\ \times
((\xi^X\otb \id_Y)\ot \id_N)m^{-1}_{T,R_{X\ot Y},N}\\
=(\id_X\odot  m_{R_{T(Y)},T,N} (\xi^Y\otb \id_N) )m^{}_{R_{T(X)},TR_Y,N} 
((\xi^X\ot \id_Y)\otb \id_N) \\
\times m^{-1}_{T,R_{X\ot Y},N}  \\
= (\id_X\odot  m_{R_{T(Y)},T,N} ) m^{}_{R_{T(X)},R_{T(Y)}T,N}((\id_X\ot
\xi^Y)(\xi^X\ot \id_Y)\otb \id_N) \\
\times
m^{-1}_{T,R_{X\ot Y},N} \\
= m_{R_{T(X)},R_{T(Y)},T\otb N} m_{R_{T(X)\ot T(Y)}, T, N}((\id_X\ot
\xi^Y)(\xi^X\ot \id_Y)\otb \id_N) \\
\times
m^{-1}_{T,R_{X\ot Y},N}. 
\end{multline*}
The second and seventh equalities by \eqref{left-modulecat1}. It follows from
\eqref{hopfmonad1} that this last expression equals \eqref{kk1}.
This proves that  $U:\No\to \No(T)$ is a lax module
functor.  Let us prove now that $(\No, U, c)$ is a $T$-equivariant $\ca$-module
category. We shall show that $\nu:U^2\to U$ is a $\ca$-module transformation. We
must
prove that
\begin{equation}\label{c-nat-trans} c_{X,N} \nu_{X\odot N}= (\id_X\odot
\nu_{N})d_{X,N},
\end{equation}
for all $X\in \ca$, $N\in \No$. Here $d_{X,N}=(\mu\otb\id_{U^2(N)})c_{T(X),T\otb
N}
U(c_{X,N})$ and $c_{X,N}$ is defined in
\eqref{mod-2}. We have
\begin{multline}\label{lhs-modt}c_{X,N} \nu_{X\odot
N}=m_{R_{T(X)},T,N}(\xi^X\otb\id_N)
m^{-1}_{T,R_X,N}(\mu\otb\id_{X\odot N})m^{-1}_{T,T,X\odot N}\\
=m_{R_{T(X)},T,N}(\xi^X\otb\id_N)(\mu\circ R_X\otb\id_N) m^{-1}_{T^2,R_X,
N}m^{-1}_{T,T,X\odot N}\\
=m_{R_{T(X)},T,N}(\xi^X(\mu\circ R_X)\otb\id_N)
m^{-1}_{T,T\circ R_X,N}
(\id_T\otb m^{-1}_{T,R_X,N}).
\end{multline}
The second equality follows from the naturality of $m$ and the third
equality follows from the associativity of $m$ \eqref{left-modulecat1}. The
right hand side of
\eqref{c-nat-trans} equals
\begin{multline*}=(\id_X\odot
(\mu\otb\id_N)m^{-1}_{T,T,N})(\mu_X\odot\id_{U^2(N)})
m_{R_{T^2(X)},T,T\otb N}(\xi^{T(X)}\otb \id_{T\otb N})\\
\times m^{-1}_{T,R_{T(X)},T\otb N}
(\id_T\otb m_{R_{T(X)},T,N})(\id_T\otb (\xi^X\ot\id_N)m^{-1}_{T,R_X,N})\\
=(\mu_X\odot (\mu\otb\id_N))(\id_{R_{T^2(X)}}\otb m^{-1}_{T,T,N})
m_{R_{T^2(X)},T,T\otb N}(\xi^{T(X)}\otb \id_{T\otb N})\\
\times m^{-1}_{T,R_{T(X)},T\otb N}
(\id_T\otb m_{R_{T(X)},T,N})(\id_T\otb (\xi^X\otb\id_N)m^{-1}_{T,R_X,N})\\
=(\mu_X\odot (\mu\ot\id_N)) m_{R_{T^2(X)},T^2,N}
m^{-1}_{R_{T^2(X)}\circ T,T, N}(\xi^{T(X)}\ot \id_{T\otb N})\\
\times m_{T\circ R_{T(X)},T,N}m^{-1}_{T,R_{T(X)}\circ T,N}
(\id_T\otb (\xi^X\ot\id_N)m^{-1}_{T,R_X,N})\\
=(\mu_X\odot (\mu\otb\id_N)) m_{R_{T^2(X)},T^2,N}
m^{-1}_{R_{T^2(X)}\circ T,T, N}(\xi^{T(X)}\otb \id_{T\otb N})\\
\times m_{T\circ R_{T(X)},T,N} ((\id_T\otb \xi^X)\ot\id_N) m^{-1}_{T,T\circ
R_X,N}
(\id_T\otb m^{-1}_{T,R_X,N})\\
=(\mu_X\odot (\mu\ot\id_N)) m_{R_{T^2(X)},T^2,N} ((\xi^{T(X)}\ot
\id_T)\otb\id_N)
((\id_T\ot \xi^X)\otb \id_N)\\
\times m^{-1}_{T,T\circ R_X,N}
(\id_T\otb m^{-1}_{T,R_X,N})\\
=m_{R_{T(X)},T,N} ( (\alpha_{\mu_X}\ot \id_T)  R_{T^2(X)}(\mu)\ot\id_N)
((\xi^{T(X)}\ot \id_T)\ot\id_N)
\\ \times ((\id_T\ot \xi^X)\ot \id_N)
 m^{-1}_{T,T\circ R_X,N}
(\id_T\ot m^{-1}_{T,R_X,N})\\
=m_{R_{T(X)},T,N} ( (\alpha_{\mu_X}\ot \id_T)  R_{T^2(X)}(\mu) (\xi^{T(X)}\circ
T)T(\xi^X)\ot\id_N)
\\ \times
 m^{-1}_{T,T\circ R_X,N}
(\id_T\ot m^{-1}_{T,R_X,N})\\
=m_{R_{T(X)},T,N} ( (\alpha_{\mu_X}\ot \id_T)  (\id_{T^2(X)}\ot\mu)
(\xi^{T(X)}\circ T)T(\xi^X)\ot\id_N)
\\ \times
 m^{-1}_{T,T\circ R_X,N}
(\id_T\ot m^{-1}_{T,R_X,N}).\\
\end{multline*}
The second equality follows from the naturality of $m$, the third equality  from
\eqref{left-modulecat1}, the fourth, fifth and sixth equalities again by the
naturality of $m$.
It remains to show that $(\alpha_{\mu_X}\ot \id_T)  (\id_{T^2(X)}\ot\mu)
(\xi^{T(X)}\circ T)T(\xi^X) =
\xi^X(\mu\circ R_X)$, but this is \eqref{hopfmonad3}.

\medbreak

2. The category $ \Fun_{\catop}(\ca^{\op}, \No)$ is a right
$\End_{\catop}(\ca^{\op})$-module category via composition of
functors. It follows from Lemma
\ref{lema-tec1} and Lemma \ref{op-equiv}
  that $\End_{\catop}(\ca^{\op})\simeq (\ca^{\op})^{\tb}
\simeq (\ca^T)^{\op}$. Thus
$ \Fun_{\catop}(\ca^{\op}, \No)$ is a left $\ca^T$-module category using
these identifications. Define the $\ca^T$-module functors
$$\Phi: \No^U\to  \Fun_{\catop}(\ca^{\op}, \No),
\quad \Psi: \Fun_{\catop}(\ca^{\op}, \No)\to \No^U$$
as follows. If $(N,s)\in \No^U$ then $\Phi(N,s)(X)=R_X\otb N$ for all
$X\in\ca$. The functor $\Phi(N,s)$ is a module functor with structure
given by
$$ c^{(N,s)}_{F,Y}:\Phi(N,s)(F\otb Y)\to F\otb \Phi(N,s)(Y),\quad
c^{(N,s)}_{F,Y}=(c^F_{-,Y}\ot\id_N)m^{-1}_{F,R_Y,N},$$
for all $(F, c^F)\in \catop$, $Y\in \ca$.
If $(G,d^G)\in \Fun_{\catop}(\ca^{\op}, \No)$ define
$\Psi(G,d^G)=(G(\uno), s^G)$, where $s^G:U(G(\uno))\to G(\uno)$ is defined
by $s^G= G(\phi)\circ (d^G_{T,\uno})^{-1}.$ Both functors $\Phi, \Psi$ are
well-defined $\ca^T$-module functors and they give an equivalence of module
categories.\epf

\begin{cor}\label{modc-eq} Let $T$ be a  right exact faithful Hopf monad  on 
$\ca$ and let $\Mo$ be an exact indecomposable $\ca^T$-module category. Then
there exists a $T$-equivariant indecomposable exact $\ca$-module category $\No$
with
simple and exact equivariant structure given by $U:\No\to \No$ such that
$\Mo\simeq \No^U$ as $\ca^T$-module categories.
\end{cor}


\pf   Let $\Mo$ be an exact indecomposable left $\ca^T$-module category.
Thus, $\Mo$ is an exact indecomposable right $(\ca^T)^{\op}$-module category.
Then, the category $\No = \Hom_{(\ca^T)^{\op}}(\ca^{\op}, \Mo)$ is an exact
indecomposable left
$\catop$-module category, see \cite[Theorem 3.31]{eo}. It follows from Theorem
\ref{equi-mc} that
$\No=\Hom_{(\ca^T)^{\op}}(\ca^{\op}, \Mo)$ is a $T$-equivariant $\ca$-module
category
with equivariant structure $U:\No\to \No$, $U(N)=\tb\otb N$, and there are
$\ca^T$-module equivalences
$$\No^U\simeq \Hom_{\catop}(\ca^{\op}, \No)\simeq \Mo. $$
Since the functor $\otb$ is biexact the functor $U$ is exact. \epf

\section{Module categories over  Hopf algebroids}\label{h-algebroids}

\subsection{ Hopf algebroids}  Let us briefly introduce the notion of
Hopf algebroid.  The reader is refered to  \cite{Bo},
\cite{BS}, \cite{KS}. Let $L$, $R$ be  algebras over $\ku$.

\begin{defi} A \emph{left bialgebroid} with base $L$ is a collection
$(H,s,t,\Delta, \epsilon)$ where $s:L\to H$, $t:L^{\opp}\to H$ are
algebra maps such that $s(l)t(l)=t(l)s(l)$ making $H$ an
$(L,L)$-bimodule:
$$l\cdot x \cdot l'=s(l)t(l')x,$$ for all $l, l'\in L, x\in H$. The remaining
data $\Delta:H\to H\ot_L H$ and 
$\epsilon:H\to L$  are $k$-linear maps which make the triple
$(H,\Delta,\epsilon)$ into a
comonoid in ${}_L\Mo_L$. Moreover, the following identities are
required to hold:
\begin{equation}\label{lb1} \Delta(x)(t(l)\ot 1)=\Delta(x) (1\ot s(l)),
\end{equation}
\begin{equation}\label{lb2} \Delta(1)=1\ot 1,
\end{equation}
\begin{equation}\label{lb3}\Delta(xy)=\Delta(x)\Delta(y),
\end{equation}
\begin{equation}\label{lb4} \epsilon(1)=1,\;\; \epsilon(x
s(\epsilon(y)))=\epsilon(xy)=\epsilon(xt(\epsilon(y))),
\end{equation}
for all $x, y\in H, l\in L$. Right bialgebroids are defined in a
similar way. See for example \cite[Definition 2.2]{BS}.
\end{defi}

\begin{defi} A \emph{Hopf algebroid} is a collection $(H_R, H_L, \Ss)$
where $H_L=(H,s_L,t_L,\Delta_L, \epsilon_L)$ is a left bialgebroid
over $L$ and $H_R=(H,s_R,t_R,\Delta_R, \epsilon_R)$ is a right
bialgebroid over $R$, $\Ss:H\to H$ is a linear map, called the
\emph{antipode}, such that
\begin{align}\label{hb1} s_L\circ \epsilon_L\circ t_R=t_R,
\qquad t_L\circ \epsilon_L\circ s_R=s_R,\\
 s_R\circ \epsilon_R\circ t_L=t_L, \qquad
 t_R\circ \epsilon_R\circ s_L=s_L,
\end{align}
\begin{equation}\label{hb2} (\Delta_L\ot\id_H)\Delta_R=(\id_H\ot
\Delta_R)\Delta_L, \quad (\Delta_R\ot\id_H)\Delta_L=(\id_H\ot
\Delta_L)\Delta_R
\end{equation}
\begin{flalign}\label{hb23}
\Ss:H\to H \text{ is both an } L   \text{ and }
R \text{ bimodule map: } \qquad \qquad  \qquad   \qquad  
\end{flalign}
$$\Ss(t_L(l)ht_L(l'))=s_L(l')\Ss(h)s_L(l),\quad
\Ss(t_R(r)ht_R(r'))=s_R(r')\Ss(h)s_R(r), $$
\begin{equation}\label{hb3} m_H\circ (\Ss\ot \id_H)\circ \Delta_L=
s_R\circ \epsilon_R,
\end{equation}
\begin{equation}\label{hb4} m_H\circ (\id_H\ot \Ss)\circ \Delta_R=
s_L\circ \epsilon_L.
\end{equation}
\end{defi}

\begin{rmk} If $(H_R, H_L, \Ss)$ is a Hopf algebroid then $R\simeq
L^{\opp}$.
\end{rmk}

If $(H_R, H_L, \Ss)$ is a Hopf algebroid the category of  finite-dimensional
left
$H$-modules is a finite tensor category.

\subsection{Hopf monads and Hopf algebroids}

Let $(H_R, H_L, \Ss)$ be a finite-dimensional Hopf algebroid.  Associated to
this Hopf algebroid, there 
is a Hopf monad $T_H$ on the category ${}_L\Mo_L$ of $L$-bimodules \cite[Section
7]{blv}. Let $L^e=L\otk L^{\opp}$.
The category ${}_L\Mo_L\simeq {}_{L^e}\Mo$ is a monoidal category with tensor
product $\ot_L$.
The functor $T_H:{}_{L^e}\Mo \to {}_{L^e}\Mo$,
$T_H(V)= H\ot_{L^e} V$, for all $V\in {}_{L^e}\Mo$ is a Hopf monad with
structure maps given by
$$ \mu_V:H\ot_{L^e} H\ot_{L^e} V\to H\ot_{L^e} V,\quad \mu_V(x\ot y\ot v)=xy\ot
v,$$
$$\eta_V: V\to H\ot_{L^e} V, \quad \eta_V(v)=1\ot v,$$
$$\xi_{V,W}: H\ot_{L^e} (V\ot_L W)\to  (H\ot_{L^e} V) \ot_L ( H\ot_{L^e} W),$$
$$ \xi_{V,W}(x\ot v\ot w)=x\_1\ot v\ot x\_2\ot w,$$
$$ \phi: H\ot_{L^e} L \to L, \quad \phi(x \ot l)=\epsilon(xs(l)),$$
for all $V, W\in {}_{L^e}\Mo$, $v\in V, w\in W$, $x,y\in H$, $l\in L$. It
follows 
 that $T_H$ is a Hopf monad. Furthermore, there is an equivalence of  tensor
categories
 $({}_L\Mo_L)^{T_H}\simeq H\rrep$.  See \cite[Corollary 5.16]{Sz}.
  Any finite tensor category is monoidally equivalent to the category of
representations of 
  a Hopf algebroid \cite[Theorem 7.6]{blv}.

\subsection{Comodule algebras over Hopf algebroids}

Let $(H,s,t,\Delta, \epsilon)$ be a Hopf algebroid. 
A \emph{left
$H$-comodule algebra} is a triple $(K,s_K,\lambda)$, where
$s_K:L\to K$ is an algebra map that makes $K$ in to a
$(L,L)$-bimodule $$l\cdot k\cdot l'= s_K(l)ks_K(l'),$$ for all
$k\in K, l,l'\in L$. A left $L$-linear map $\lambda:K\to H\ot_L
K$, such that
\begin{equation}\label{lca1} (\epsilon\ot\id_K)\lambda=\id_K,\quad
(\id_H\ot \lambda)\, \lambda= (\Delta\ot\id_K)\, \lambda,
\end{equation}
\begin{equation}\label{lca2} \lambda(K)\subseteq H\times_L K=\{x\in
H\ot_L K: \forall\; l\in L, x(t(l)\ot 1)= x(1\ot s_K(l))\}
\end{equation}
\begin{equation}\label{lca3} \lambda(1)=1\ot 1,
\;\;\lambda(x)\lambda(y)=\lambda(xy), \,
\text{ for all } x, y\in K.
\end{equation}

Equation \eqref{lca3} makes sense in view of  axiom
\eqref{lca2}. We shall use Sweedler's notation:
$\lambda(k)=k\_{-1}\ot k\_0$, for all $k\in K$.

\begin{defi} We say that a  left $H$-comodule algebra $(K,s_K,\lambda)$ is
$H$-\emph{simple} if it has
no non-trivial $H$-costable ideals.
\end{defi}

\begin{exa} \begin{enumerate}\item[(1)] $(H,s, \Delta)$ is a left $H$-comodule
algebra.
\item[(2)] $(L,s_L, \lambda_L)$ is a left $H$-comodule algebra,
    where $s_L=\id_L$ and $ \lambda_L:L\to H\ot_L L$ is
    \emph{trivial}, that is $ \lambda_L(l)= s(l)\ot 1,$ for any $l\in
    L$.\end{enumerate}
\end{exa}

\subsection{ Module categories over Hopf algebroids}

Let $(K,s_K,\lambda)$ be a left $H$-comodule algebra. If $M$ is a
left $K$-module then $M$ is a left $L$-module via $s_K$. For any
$X\in H\rrep$ the tensor product $X\ot_L M$ is a left $K$-module
with action given by
\begin{align} k\cdot (x\ot m)=k\_{-1}\cdot x\ot k\_{0}\cdot m,
\end{align}
for all $k\in K, x\in X, m\in M$. As a consequence of \eqref{lca2} this action
is
well-defined.

\begin{prop} Suppose $L$ is semisimple. Then the category ${}_K\Mo$ is a
$H\rrep$-module as follows. The action is given by
$$\otb:H\rrep\times {}_K\Mo\to {}_K\Mo,\quad X\otb M=X\ot_L M,$$
for all $X\in H\rrep,$ $M\in {}_K\Mo$. The associativity and unit
isomorphisms are canonical. \qed 
\end{prop}

Observe that the exactness assumption on the module category is only needed in Theorem \ref{mc-halgd}
below, and the assumption in that theorem is that $L$ is simple (hence
semisimple).

For any left $H$-comodule algebra $(K,s_K,\lambda)$ we shall introduce a
$T_H$-equivariant structure on the module category ${}_L\Mo$.
Define $U_K: {}_L\Mo \to {}_L\Mo$, $U_K(V)=K\ot_L V$, for all $V\in {}_L\Mo$.
For any $X\in H\rrep, V\in {}_L\Mo$ define
$$c_{X,V}:K\ot_L X\ot_L V\to (H\ot_{L^e} X) \ot_L K\ot_L V,$$
$$ c_{X,V}(k\ot x\ot v)=k\_{-1}\ot x\ot k\_0\ot v.$$
Let $\nu:U^2_K\to U_K$, $u:\Id\to U_K$, be defined as follows. For any $V\in
{}_L\Mo$
$$\nu_V: K\ot_L K\ot_L V\to K\ot_L V, \quad \nu_V(k\ot k'\ot v)=kk'\ot v,$$
$$u_V:V\to K\ot_L V, \quad u_V(v)=1\ot v.$$
It readily follows that all maps described above are well-defined.

\begin{prop}\label{propert-comdalg} The following assertions hold:
\begin{itemize}
 \item[1.] The  triple $({}_L\Mo, U_K, c)$ is a $T_H$-equivariant structure.
 \item[2.] There is an
equivalence  ${}_K\Mo \simeq ({}_L\Mo)^{U_K}$ of $H\rrep$-module categories.
 \item[3.] $({}_L\Mo, U_K, c)$ is simple if and only if $K$ is $H$-simple.
\end{itemize}

\end{prop}
\pf 1. The proof that $\nu_V:  U_K\circ  U_K\to  U_K$ is a module natural
transformation is straightforward.
One has to observe that the module structure on the functor $U_K\circ  U_K$,
given in
Lemma \ref{structure of square},  is
$$d_{X,V}:  K\ot_L K\ot_L (X\ot_L V)\to (H\ot_{L^e} X)\ot_L ( K\ot_L K\ot_L V),
$$
$$d_{X,V}(h\ot g\ot x\ot v)=h\_{-1}g\_{-1}\ot x\ot h\_0\ot g\_0\ot v,$$
for all $h,g\in K, x\in X, v\in V$. The proof of part 2 is straightforward.

3. Assume  there exists a non-trivial $H$-costable ideal $I\subseteq K$.
The canonical projection 
$K\to K/I$ induces a natural morphism $U_{K/I} \to U_K$ that it is not an
isomorphism. Hence 
$({}_L\Mo, U_K, c)$ is not simple. If $K$ is $H$-simple then, by the argument in
the proof of \cite[Proposition 1.18]{AM},
the module category $({}_L\Mo)^{U_K}$ is simple. The result follows from part 2 and 
Theorem \ref{U-equivariant modcat} (3). \epf

\begin{teo}\label{mc-halgd} Assume that the algebra $L$ is simple. Any exact indecomposable
module category over $H\rrep$ is equivalent to ${}_K\Mo$ for some
$H$-simple left $H$-comodule algebra $(K,s_K,\lambda)$.
\end{teo}
\pf It follows from Corollary \ref{modc-eq} that any
exact indecomposable
module category over $H\rrep\simeq ({}_L\Mo_L)^{T_H}$ is of the form $\No^U$
for some exact indecomposable ${}_L\Mo_L$-module category $\No$ and a simple
$T_H$-equivariant structure $(\No, U, c)$. Since $L$ is simple as an algebra,
the only
exact indecomposable ${}_L\Mo_L$-module category is ${}_L\Mo$, hence $\No\simeq
{}_L\Mo$. Since $U:
{}_L\Mo\to {}_L\Mo$ is an exact functor, there exists an $L$-bimodule $K$ such
that
$U(V)=K\ot_L V$ for all $V\in {}_L\Mo$. Let us prove that $K$ is a left
$H$-comodule algebra.

Define $s_K:L\to K, s_K(l)=l\cdot k$ for all $l\in L, k\in K$. For any $X\in
{}_L\Mo_L, V\in {}_L\Mo$
the module structure on $U$ is given by the map
$$c_{X, V}: K\ot_L (X \ot_L V)\to (H\ot_{L^e} X)\ot_L K \ot_L V.$$
Define $\lambda:K\to H\ot_L K$ by
$\lambda(k)=c_{L,L}(k\ot 1\ot 1)$
for all $k\in K$. Equations \eqref{lca1} follows from \eqref{modfunctor1} and
\eqref{modfunctor2}.
Since $U$ is a monad, there exists a module transformation $\mu:U^2\to U$. The
algebra
structure on $K$ is given by $\mu_L:K\ot_L K\to K$. Equations \eqref{lca3}
follows since
$\mu$ is a module transformation. It follows from Proposition
\ref{propert-comdalg} that there
is a module equivalence
$\No^U \simeq {}_K\Mo.$
\epf

\section{ A 2-categorical interpretation}\label{2-cat}

Let us briefly recall the notion of 2-monad in a 2-category. The reader is
refered to \cite{KSt}, \cite{St}.  A \emph{2-category } consists of
\begin{itemize}
 \item a class of objects or 0-cells $Obj(\Bc)$;
\item for a pair of 0-cells $A, B$ a category $\Bc(A,B)$. Objects in $\Bc(A,B)$
are 1-cells and
morphisms are called 2-cells;
\item for any 0-cell $A$ there is a 1-cell $I_A\in \Bc(A,A)$;
\item for any 0-cells $A,B, C$ a functor
$$\circ^{A,B,C} :\Bc(B,C)\times \Bc(A,B)\to\Bc(A,C),$$ such that
it is associative and unitary. Sometimes we shall omit the superscript and
denote
the functor $\circ^{A,B,C}$ simply as $\circ$.
\end{itemize}

If $\Bc, \Bc'$ are 2-categories, a \emph{2-functor} $F:\Bc\to \Bc'$ consists of
the following data:
\begin{itemize}
 \item an  assignment  $F: Obj(\Bc)\to Obj(\Bc')$;
\item for any 0-cells $A, B$ a functor $F_{A,B}: \Bc(A,B) \to \Bc'(F(A),F(B))$
such that
for any 0-cells $A, B, C$ and 1-cells $X\in \Bc(B,C), Y\in \Bc(A,B)$
$$F_{A,C}(X\circ Y)= F_{B,C}(X)\circ F_{A,B}(Y), \quad F_{A,A}(I_A)=I'_{F(A)}.$$

\end{itemize}
If $X \in  \Bc(A,B)$ is a 1-cell or a 2-cell we shall sometimes denote
$F_{A,B}(X)$ simply by
$F(X)$ avoiding subscripts.

If $\Bc, \Bc'$ are 2-categories and $F, G: \Bc\to \Bc'$ are 2-functors, a\emph{
2-natural
transformation} $\theta:F\to G$ consists of the following data:
\begin{itemize}
 \item  for any 0-cell $A\in Obj(\Bc)$ a 1-cell $\theta_A\in \Bc'(F(A),G(A))$;

\item for any 0-cells $A, B\in Obj(\Bc)$ and any 1-cell $X\in  \Bc(A,B)$ a
natural
transformation
$$ \theta_X: \theta_B\circ F_{A,B}(X)  \to  G_{A,B}(X)  \circ \theta_A, $$
such that for any 0-cell $A$ and any 1-cells $X, Y$
$$\theta_{X \circ Y}=  (\id\circ \theta_Y)(\theta_X\circ \id), \quad
\theta_{I_A}=\id_{\theta_A}.$$
\end{itemize}

\begin{defi}\label{defi:2-monad} 1. Let $\Bc$ be a 2-category. A \emph{2-monad}
over $\Bc$ is a strict
monad, in the sense of \cite[ Definition  5.4.1]{Be}, inside the 2-category of
2-categories. Explicitly, a 2-monad is a
collection $(\Tt, \mu, \eta)$ where $\Tt: \Bc\to \Bc$ is a 2-functor, $\mu: \Tt^2\to
\Tt$ and $\eta:\Id\to \Tt$
 are 2-natural transformations satisfying
$$ \mu_A\circ \mu_{\Tt(A)}= \mu_A\circ \Tt_{\Tt^2(A), \Tt(A)}(\mu_A),
\quad \mu_A\circ \Tt_{A,\Tt(A)}(\eta_A)=I_{\Tt(A)}=\mu_A\circ  \eta_{\Tt(A)},$$
 $$\mu_X \mu_{\Tt(X)}= \mu_X \Tt(\mu_X),\quad       \mu_X \eta_{\Tt(X)}=\id_{\Tt(X)}=
\mu_X \Tt(\eta_X),$$
for any 0-cells $A, B,$ and any 1-cell $X\in  \Bc(A,B)$.
\end{defi}

\begin{exa} Let $\ca$ be a strict monoidal category. Associated to
$\ca$ there is a 2-category $\cab$ with a single object $0$. Namely,
$\cab(0,0)=\ca$ and the composition is the monoidal product in $\ca$. A bimonad
$T:\ca\to \ca$, with strict comonoidal structure, gives rise to a 2-monad
$\underline{T}:\cab \to \cab$;
$\underline{T}(0)=0$ and $\underline{T}_{0,0}=T$.
\end{exa}

If $\ca$ is a tensor category, we shall denote by ${}_\ca\Mod$, respectively
${}_\ca\Mod^{lax}$, the 2-categories whose 0-cells are  left  $\ca$-module
categories,
1-cells are $\ca$-module functors (respectively, lax $\ca$-module functors) and
2-cells are $\ca$-module natural transformations.

Let  $T$ be a Hopf monad  on $\ca$. Define the 2-functor
$\Tt:{}_\ca\Mod^{lax}\to {}_\ca\Mod^{lax}$ as follows. For any
$\Mo\in {}_\ca\Mod^{lax}$, set $\Tt(\Mo)=\Mo(T)$, see Subsection \ref{subsection:t-eq}.
For any pair $\Mo, \No\in {}_\ca\Mod^{lax}$,  set $$\Tt_{\Mo, \No}: \Funl_\ca(\Mo, \No)\to \Funl_\ca(\Mo(T), \No(T))$$ to be the functor
defined by 
$$\Tt_{\Mo, \No}(G,d)=(G, d_{T(-), -}).$$

\begin{lema}\label{monad-to-2-monad} The 2-functor $\Tt:{}_\ca\Mod^{lax}\to
{}_\ca\Mod^{lax}$ has a structure of
2-monad.
\end{lema}
\pf We shall define 2-natural transformations $\mu:\Tt^2\to \Tt, \eta: \Id\to \Tt$
such that
$(\Tt, \mu, \eta)$ is a 2-monad. Note that, abusing of the notation, we are
denoting with the same
symbols the  2-monad  structure on $\Tt$ and the monad structure on $T$.
\medbreak

For any
$\Mo\in {}_\ca\Mod^{lax}$ define $\eta_{\Mo}\in \Funl_\ca(\Mo, \Mo(T))$ the lax
$\ca$-module functor
as
$\eta_{\Mo}=(\Id_\Mo, \eta\otb \id)$. Here the module structure of the identity
functor is
given by
$$\eta_X\otb \id_M: X\otb M\to T(X)\otb M,$$
for any $X\in \ca, M\in  \Mo$. It follows from \eqref{hopfmonad11} that
$(\Id_\Mo, \eta\otb \id)$ is indeed a module functor. To give a structure of
2-natural transformation on $\eta$, for any $(G,d)\in \Funl_\ca(\Mo, \No)$, we
must define natural
transformations
$$\eta_{(G,d)}: \eta_\Mo\circ (G,d) \to  \Tt_{\Mo, \No}(G,d) \circ \eta_\No. $$
Since both functors are equal, we let $\eta_{(G,d)}$ to be the identity natural
transformation.
Now, let us define the 2-natural transformation $\mu:\Tt^2\to \Tt$.
For any $\Mo\in {}_\ca\Mod^{lax}$, let $\mu_\Mo=(\Id_\Mo, \mu\otb  \id)$. It
follows
from \eqref{hopfmonad3} that  this functor is indeed
a module functor.  For any $\Mo, \No\in {}_\ca\Mod^{lax}$, $(G,d)\in
\Funl_\ca(\Mo, \No)$, we must define natural
transformations
$$\mu_{(G,d)}: \mu_\Mo\circ \Tt^2_{\Mo, \No}(G,d) \to  \Tt_{\Mo, \No}(G,d)\circ
\mu_\No. $$
Since both functors are equal, we define $\mu_{(G,d)}$ the identity natural
transformation.
Conditions of Definition \ref{defi:2-monad} are readily verified.\epf

 In the next subsection we shall give an interpretation of the
process of equivariantization of module categories in the form of a 2-category equivalence between the 2-category ${}_\ca\Mod^{lax}$ with an appropriate equivariantization of the  2-category of $T$-equivariant lax
$\ca$-module categories, that we define next. 

\begin{defi} Let $\ca$ be a tensor category and $T:\ca\to \ca$ be a
Hopf monad.  The 2-category  ${}_\ca\eqmod$ of $T$-equivariant lax
$\ca$-module categories is defined as follows: 
  0-cells are $T$-equivariant $\ca$-module categories $(\Mo, U, c)$. If
$(\Mo, U, c)$, $(\widetilde{\Mo}, \widetilde{U}, \widetilde{c})$ are
$T$-equivariant $\ca$-module categories, 1-cells are pairs
$(G,\theta):(\Mo, U, c)\to (\widetilde{\Mo}, \widetilde{U}, \widetilde{c}) $,
where $G:\Mo\to \widetilde{\Mo}$ is a $\ca$-module functor and 
$\theta:  \widetilde{U}\circ G\to \widehat{T}(G)\circ U$ is a natural
transformation satisfying condition 
\eqref{2-nat-f-eq},  that is,
$$\theta_M \widetilde{\nu}_{G(M)}=G(\nu_M)\theta_{U(M)}\widetilde{U}(
\theta_M),\quad
 \theta_M \widetilde{u}_{ G(M)}=G(u_M),
$$
for all $M\in \Mo$.
 If $(H, \chi)$ is another 1-cell, a 2-cell 
$\alpha:(G,\theta) \Rightarrow(H, \chi)$ is a
$\ca$-module natural transformation $\alpha: G\to H$ satisfying condition 
\eqref{2-nat-f-eq2}, that is,
$$ \alpha_{U(M)}\theta_M=\chi_M  \widetilde{U}(\alpha_M),$$ for all $M\in \Mo$.
\end{defi}

\subsection{Equivariantization of 2-categories}

 Let $\Bc$ be a 2-category and let $(F, \mu, \eta):\Bc\to \Bc$ be a 2-monad on $\Bc$. 
 We start by giving a description of the 2-category of $\Bc^F$ of $F$-equivariant objects in $\Bc$.

The  horizontal  composition  of $1$ or $2$-cells  in the 2-category $\Bc$ will be denoted by $\circ$,
omitting the
superscripts, and the  vertical  composition  of $2$-cells  will be indicated by juxtaposition of morphisms.

\begin{defi} An \emph{equivariant object}  (or \emph{equivariant 0-cell})  is a collection $(A, U, \nu, u)$
where
\begin{itemize}
 \item $A$ is a 0-cell in $\Bc$;

\item $U:A\to F(A)$ is a 1-cell in $\Bc$;

\item $\nu: \mu_A\circ F(U)\circ U\Rightarrow U$, $u:\eta_A\Rightarrow U$ are
2-cells, such that
\begin{equation} \nu (\id_{\mu_A\circ F(U)} \circ \nu)(\id_{\mu_A}\circ \mu_U\circ
\id_{F(U)\circ U})
=\nu (\id_{\mu_A}\circ F(\nu)\circ \id_{U}),
\end{equation}
\begin{equation}  \nu (\id_{\mu_A}\circ F(u)\circ \id_U)=\id_U=  \nu
(\id_{\mu_A\circ F(U)}\circ  u) (\id_{\mu_A}\circ \eta_U).
\end{equation}
\end{itemize}
\end{defi}

\medbreak Let $(A, U, \nu, u)$, $(\widetilde{A}, \widetilde{U}, \widetilde{\nu},
\widetilde{u})$ be $F$-equivariant objects.     An  \emph{equivariant 1-cell} 
$(\theta, \theta^0):(A, U, \nu, u) \to (\widetilde{A}, \widetilde{U},
\widetilde{\nu}, \widetilde{u})$
consists of
\begin{itemize}
 \item  a 1-cell $\theta:A\to \widetilde{A}$;

 \item a 2-cell $\theta^0: \widetilde{U}\circ \theta\Rightarrow F(\theta)\circ
U$,
\end{itemize}
satisfying  the following conditions:
\begin{equation*} \theta^0 (\widetilde{u}\circ \id_\theta)=(\id_{F(\theta)}\circ
u) \eta_\theta,
\end{equation*}
\begin{equation*}(\id_{F(\theta)}\circ \nu)(\mu_\theta \circ \id_{F(U)\circ
U})(\id_{\mu_{ \widetilde{A}}}
\circ F(\theta^0)\circ \id_U)
(\id_{\mu_{ \widetilde{A}}\circ F( \widetilde{U})}\circ \theta^0 )= \theta^0
(\widetilde{\nu} \circ \id_\theta).
\end{equation*}

Let $(A, U, \nu, u)$ and $(\widetilde{A}, \widetilde{U}, \widetilde{\nu},
\widetilde{u})$ be 
 $F$-equivariant objects and let 
$(\theta, \theta^0)$, $(\chi, \chi^0):(A, U, \nu, u) \to (\widetilde{A},
\widetilde{U}, \widetilde{\nu}, \widetilde{u})$
be  equivariant 1-cells.  An  
\emph{equivariant 2-cell} $(\theta, \theta^0) \Rightarrow (\chi, \chi^0)$ is a 2-cell  $\alpha: \theta \Rightarrow \chi$  such that 
\begin{equation} \chi^0 (\id_{\widetilde{U}}\circ \alpha)=(F(\alpha)\circ \id_U)
\theta^0.
\end{equation}

 The proof of the following proposition is left to the reader.

\begin{prop} Equivariant 0-cells, equivariant 
1-cells  and equivariant 2-cells form a 2-category 
with respect to composition of  1-cells  $(\theta, \theta^0):(A, U, \nu, u) \to (\widetilde{A}, \widetilde{U},
\widetilde{\nu}, \widetilde{u})$
and $(\chi, \chi^0): (A', U', \nu', u') \to (A, U, \nu, u)$  defined by
\begin{equation}(\theta, \theta^0) \circ (\chi, \chi^0)=(\theta \circ \chi,
(\id\circ \chi^0)(\theta^0\circ\id)),
\end{equation}
and  vertical and horizontal compositions  of equivariant 2-cells given as the corresponding compositions in the 2-category
$\Bc$. \qed
\end{prop}

This 2-category will be called the \emph{2-category of $F$-equivariant objects} in $\Bc$ and will be denoted by $\Bc^F$.


\begin{rmk}\label{ilf} An $F$-equivariant object in $\Bc$ could be explained alternatively as a
monad, in the sense of
\cite[Definition 5.4.1]{Be}, inside the Kleisli 2-category associated with
$F$, and the 2-category $ \Bc^F$ as the 2-category of monads
 inside the Kleisli 2-category.
\end{rmk}

The following result is a straightforward application of the definitions.
\begin{prop}\label{equiv-eqmod} Let $\ca$ be a tensor category and $T:\ca\to \ca$ be a
Hopf monad. Let   $\Tt:{}_\ca\Mod^{lax}\to {}_\ca\Mod^{lax}$
be the 2-monad associated to $T$ des\-cribed in Lemma \ref{monad-to-2-monad}.
There exists a
2-equivalence of 2-categories
$${}_\ca\eqmod \simeq   ({}_\ca\Mod^{lax})^\Tt. $$ \qed 
\end{prop}

\section{ Module categories and exact sequences }\label{mc-exact}

Let $\ca$, $\Do$ be tensor
categories over $\ku$. Recall that a \emph{normal} tensor functor $F: \ca \to \Do$
is a tensor functor such that for any object $X$ of $\ca$, there
exists a subobject $X_0 \subset X$ such that $F(X_0)$ is the
largest trivial subobject of $F(X)$.

If the functor $F$ has a right adjoint $R$, then $F$ is normal if and only if
$R(\uno)$ is a trivial object of $\ca$ \cite[Proposition 3.5]{BN}.

\medbreak Let $\ca', \ca, \ca''$ be tensor categories over $\ku$. A sequence of
tensor functors
\begin{equation}\label{exacta-fusion}\xymatrix{\ca' \ar[r]^f & \ca \ar[r]^F &
\ca''}
\end{equation}
is called and \emph{exact sequence of tensor categories} if the
following hold:
\begin{itemize}
\item  The tensor functor $F$ is dominant and normal;
\item The tensor functor $f$ is a full embedding;
\item The essential image of $f$ is $\KER_F$.
\end{itemize}
See \cite{BN}. Here, $\KER_F$ is the full tensor subcategory
$F^{-1}(\langle \uno \rangle) \subseteq \ca$  of objects $X$ of
$\ca$ such that $F(X)$ is a trivial object of $\ca''$.

\medbreak Suppose \eqref{exacta-fusion} is an exact sequence of tensor
categories. 
Since the functor $F$ is normal, then it induces a fiber functor  $\omega_F:
\ca' \to
\vect_\ku$ in the form 
$\omega_F(X) = \Hom_{\ca''}(\uno, Ff(X))$. 

The \emph{induced Hopf algebra} $H$ of the exact sequence
\eqref{exacta-fusion} is defined as
the coend of the fiber functor $\omega_F: \ca' \to \vect_k$,  that is, $H =
\int^{X \in \ca'}\omega_F(X)^\vee \otimes \omega_F(X)$. In particular, we have
an equivalence
of tensor categories $\ca' \simeq \corep H$.    See \cite[Subsection 3.3]{BN}.

\medbreak Recall that a $\ku$-linear right exact Hopf monad $T$ on a tensor
category
$\ca''$ is called \emph{normal} if $T(\uno)$ is a trivial object of $\ca''$. By
\cite[Theorem 5.8]{BN} exact sequences \eqref{exacta-fusion} with finite dimensional induced Hopf
algebra $H$ are classified by normal faithful right exact $\ku$-linear Hopf
monads $T: \ca'' \to \ca''$, such that the Hopf monad of the restriction of $T$
to the trivial subcategory of $\ca''$ is isomorphic to $H$.

\subsection{The Hopf monad of a Hopf algebra extension} Consider an exact
sequence of finite dimensional Hopf algebras
\begin{equation}\label{sec-hopf}\ku \lto K \overset{i}\lto H \overset{\pi}\lto
\overline H \lto \ku. \end{equation}
In view of \cite[Proposition 3.9]{BN} \eqref{sec-hopf} induces an exact
sequence of finite tensor categories
\begin{equation}\label{sec-mod}\rep\overline H \overset{\pi^*}\lto \rep H
\overset{i^*}\lto \rep K. \end{equation}

In this subsection we shall give an
explicit description of the normal Hopf monad $T$ on $\rep K$ corresponding to
the exact sequence \eqref{sec-mod} in terms of the cohomological data
classifying the Hopf algebra extension \eqref{sec-hopf}.

\medbreak As a consequence of the Nichols-Zoeller freeness theorem, the exact
sequence \eqref{sec-hopf} is cleft. Therefore there exist maps
\begin{align}& . : \overline H \otimes K \to K, \qquad \sigma: \overline H
\otimes \overline H \to K, \\
& \rho: \overline H \to \overline H \otimes K, \qquad \tau: \overline H \to K
\otimes K,  \end{align} obeying the compatibility conditions in \cite[Theorem
2.20]{AD},  such that $H$ is isomorphic as a Hopf algebra to the bicrossed
product $K ^\tau\#_\sigma \overline H$. Recall that the structure of $K
^\tau\#_\sigma \overline H$ is determined by the data $(., \sigma, \rho, \tau)$
as follows:
\begin{align*} & a\#x \, . \, b\#y = a(x_{(1)}.b) \sigma(x_{(2)}, y_{(1)}) \#
x_{(3)}, y_{(2)},  \\
& \Delta(a\# x) = \Delta(a) \tau(x_{(1)}) \, (1 \# (x_{(2)})_{(0)} \otimes
(x_{(2)})_{(1)} \# x_{(3)}), \\ & 1_{K ^\tau\#_\sigma \overline H} = 1 \# 1,
\quad  \epsilon(a\# x) = \epsilon (a) \epsilon(x),
\end{align*} for all $a, b \in K$, $x, y \in \overline H$, where we use
Sweedler's notation $\rho(x) = x_{(0)} \otimes x_{(1)} \in \overline H \otimes
K$, for every $x \in \overline H$.

\medbreak In what follows we shall use the identifications $H = K ^\tau\#_\sigma
\overline H$,  $K \simeq K \# 1 \subseteq H$ and $\pi = \epsilon \otimes \id : H
\to \overline H$.
In this way, the normal tensor functor $F = i^*: \rep H \to \rep K$ corresponds
to the restriction functor $\Res^H_K$.

It is well known that the induction functor $L = \Ind^H_K: \rep K \to \rep H$ is
left adjoint
of $F$, where for every right $K$-module $W$, $\Ind^H_K(W) = W \otimes_KH$.
 It is clear from the formula defining the multiplication of $K ^\tau\#_\sigma
\overline H$
 that $H \simeq K \otimes \overline H$ as left  $K$-modules, where $K$ acts by
left multiplication in the first tensorand on $K \otimes \overline H$. Hence we
obtain  natural isomorphisms
$$r_W: L(W) \simeq W \otimes \overline H, \quad r_W(w \otimes a \# x) = (w
\leftharpoonup a) \otimes x,$$ for every finite-dimensional right $K$-module
$W$,  $w \in W$, $a\in K$, $x \in \overline H$, where $\leftharpoonup: W \otimes
K \to W$ denotes the $K$-module structure on $W$.

\begin{prop}\label{monad-ext} The normal Hopf monad $T: \rep K \to \rep K$
associated to the exact sequence \eqref{sec-mod} is given by $T(W) = W \otimes
\overline H$, where the $K$-action is defined as $$(w \otimes x) \leftharpoonup
a = w(x_{(1)}.a) \otimes x_{(2)},$$ for all $w \in W$, $x \in \overline H$,
$a\in K$. For all $W, U \in \rep K$, the multiplication $\mu_W: W \otimes
\overline H \otimes \overline H \to W \otimes \overline H$, counit $\eta_W: W
\to W \otimes \overline H$, and comonoidal structure $\xi_{W, U}: W \otimes U
\otimes \overline H \to W \otimes \overline H \otimes U \otimes \overline H$,
$\phi: \ku \otimes \overline H \to \ku$, of $T$ are determined, respectively, as
follows:
$$\begin{aligned} \mu_W(w \otimes x\otimes y) & = w \leftharpoonup
\sigma(x_{(1)}, y_{(1)}) \otimes x_{(2)} y_{(2)},\\
 \eta_W(w) & = w \otimes 1, \\
 \xi_{W, U} (w \otimes u \otimes x) & = (w \otimes u) \leftharpoonup
\tau(x_{(1)})(1 \otimes (x_{(2)})_{(1)}) \, \otimes (x_{(2)})_{(0)} \otimes
x_{(3)}, \\
 \phi & = \id \otimes \epsilon_{\overline H}: \ku \otimes \overline H \to \ku,
 \end{aligned}$$ for all $w \in W$, $u \in U$, $x, y\in \overline H$. \end{prop}

 \begin{proof} Since $F: \rep H \to \rep K$ is a strict strong monoidal functor,
then the monad $T = FL$ of the adjunction $(L, F)$ is a bimonad with the
prescribed structure; see \cite[Theorem 2.6]{BV}. Furthermore, $T$ is a Hopf
monad, by \cite[Proposition 3.5]{blv}.
Note that, thus defined, $T$ is normal, faithful and right exact and the induced
Hopf algebra of $T$ in the sense of \cite[Section 5]{BN} coincides with
$L(\ku)^* = (\overline  H)^*$. This proves the proposition. \end{proof}

\begin{rmk} Consider the case where  the exact sequence \eqref{sec-hopf}
 is \emph{cocentral} or, equivalently, $\overline H$ is isomorphic to the
group algebra $\ku G$ of a finite group $G$ and the weak coaction $\rho$ is
trivial.
 It is known that the cohomological data $., \sigma, \tau$ give rise to an
action of $G$
 on the category $\rep K$ by tensor autoequivalences and the equivariantization
$(\rep K)^G$ is equivalent to $\rep H$ as tensor categories \cite[Subsection
3.3]{ext-ty}.
In this case, the normal Hopf monad $T$ associated to the Hopf algebra extension
by
Proposition \ref{monad-ext} coincides with the Hopf monad of the corresponding
group
 action given by \cite[Theorem 5.21]{BN}. \end{rmk}

\subsection{The category $\caT$ when $T$ is normal} An exact sequence of
tensor categories
\begin{equation}\label{exacta-gral}\ca' \lto \ca \overset{F}\lto \ca''
\end{equation} is called \emph{perfect} if 
$F$ is a perfect tensor functor, that is, $F$ admits an exact right adjoint $R:
\ca'' \to \ca$. 

Suppose that $\ca'$ is a finite tensor category. Then 
\eqref{exacta-gral} corresponds to a normal faithful $\ku$-linear right exact
Hopf monad $T$ on $\ca''$. In this case the sequence is perfect if and only if
$T$ is an exact endofunctor of $\ca''$.

\medbreak Let $(A, \sigma) \in \mathcal Z(\ca)$ be the induced central
algebra of $F$. Since $A  = R(\uno)$, then the normality of $F$ is equivalent to
the assumption that $A$ belongs to $\KER_F$.

We shall use the identifications $\ca'' = \ca_A$ and $F = F_A: \ca \to
\ca_A$ is the free $A$-module functor. 
Since $F_A$ is a tensor functor, then $F_A(A) = A \otimes A$ is an algebra in
$\ca_A$ with multiplication $m \otimes
\id_A$. 

An object of ${}_{F(A)}(\ca_A)$ is a right $A$-module $Y$ in $\ca$
endowed with a morphism  $F_A(A) \otimes_A Y \to Y$ in
$\ca_A$. Note that for all object $Y$ of $\ca_A$ we have a
canonical isomorphism $F_A(A) \otimes_A Y \simeq A \otimes Y$ in $\ca_A$, where
the right $A$-module structure on $A \otimes Y$ is given by the right action of
$A$ on $Y$.    Hence a morphism $F_A(A) \otimes_A Y \to Y$ is uniquely
determined by a morphism $A \otimes Y \to Y$ of right $A$-modules in $\ca$.

We obtain in this way an equivalence of categories ${}_{F(A)}\ca''
\simeq {}_A\ca_A$. 
Since $F$ is normal, $F(A)$ is a trivial object of $\ca''$, and this restricts
in addition to an equivalence ${}_{F(A)}\langle \uno \rangle \simeq
{}_A(\KER_F)_A$.

\begin{prop}\label{klinear} Let $\ca$ be a finite tensor category and let $T$ be a normal
faithful $\ku$-linear exact Hopf monad on $\ca$ with induced Hopf algebra
$H$.
Then there is an equivalence of $\ku$-linear categories $(\ca \rtimes T)^{\op}
\simeq H\rrep \boxtimes \, \ca$.
\end{prop}

\begin{proof}   The functor $\Fc: \ca^T \to \ca$ gives rise to an exact sequence
of finite tensor categories $\ca' \to \ca^T \overset{\Fc}\to \ca$ such that
$\ca' \simeq \corep H$.
Since $T$ is exact by assumption, then $\Fc$ is a perfect tensor functor and
there is an equivalence of tensor categories $\ca \rtimes T \simeq _A(\ca^T)_A$,
where $(A, \sigma)$ is the induced central algebra of $T$.
 
From the previous discussion, we have that $_A(\ca^T)_A \simeq {}_{\Fc(A)}\ca$.
  Since $\Fc(A)$ is a trivial object of $\ca$, then it follows from
\cite[Proposition 5.11]{deligne} that the tensor product $\otimes: \langle \uno
\rangle \times \ca \to \ca$ induces an equivalence of $k$-linear categories
${}_{\Fc(A)}\langle \uno \rangle \boxtimes \ca \to {}_{\Fc(A)}\ca$.  

To finish the proof we observe that there is an equivalence of
tensor categories ${}_A\ca'_A \simeq  H\rrep$. Indeed, the normality of $\Fc$
implies that $\Fc$ induces a fiber 
functor $\Fc: \ca' \to \langle \uno \rangle$, whose coend is isomorphic to $H$
and whose induced central algebra is
isomorphic to $A$. Hence $\vect_\ku \simeq {\ca'}_A$.
By \cite[Theorem 5]{ostrik}  we get equivalences of
tensor categories ${}_A{\ca'}_A \simeq \End_{\ca'}(\mathcal M) \simeq H\rrep$. 
This finishes the proof of the proposition. 
\end{proof}

\begin{exa} (Hopf algebra exact sequences.) Consider an exact sequence of Hopf algebras $\ku \to K \to H \to
\overline H \to \ku$ and assume that $K$ is finite dimensional. Then $H$ is free
as
a left (or right) module over $K$ and in particular the sequence is cleft
\cite[Theorem 2.1 (2)]{schneider}. By \cite[Proposition 3.9]{BN} we have an
exact sequence of tensor categories
\begin{equation}\label{sec-comod} \corep K \to \corep H \to \corep \overline H.
\end{equation}
Moreover, since $\corep K$ is a finite tensor category, then the exact sequence
\eqref{sec-comod} is determined by a normal faithful Hopf monad $T$ on $\corep
\overline H$.

Observe that $K$ has a natural algebra structure in the category $\corep H$ and,
by cleftness, there is an equivalence of $\corep H$-module categories $\corep
\overline H \simeq (\corep H)_K$. Hence we obtain an equivalence of tensor
categories $(\corep H) \rtimes T \simeq {}_K(\corep H)_K$.
The last category is equivalent to the category of comodules over the
coquasibialgebra $(K^* \bowtie \overline H, \varphi)$, where $\varphi$ is an
associated Kac 3-cocycle \cite[Section 6]{schauenburg}. Thus we get an
equivalence of tensor categories $(\corep H) \rtimes T \simeq \corep (K^* \bowtie
\overline H, \varphi)$.
\end{exa}

\begin{exa} (Equivariantization exact sequences.) 
Let $G$ be a finite
group and let $\rho: \underline{G}\to \underline{\text{Aut}}_\otimes(\ca)$ be an action by tensor autoequivalences of $G$ on the finite tensor category $\ca$. Let also $\ca^G$ denote the corresponding equivariantization.  

The $G$-action gives rise to a normal Hopf monad $T = T^\rho$ on $\ca$ in such a way that $\ca^{T^\rho} \simeq \ca^G$ as tensor categories over $\ca$ (see Example \ref{hm-equiv}). As discussed in \cite[Subsection 5.3]{BN}, we obtain in this way a   (central) exact sequence of tensor categories $$\Rep G \to \ca^G \to \ca.$$ 

Suppose that $\ca$ is a fusion category. The category $\ca \rtimes T^\rho$ and its module categories were studied by Nikshych in \cite{nik}.  It follows from \cite[Proposition 3.2]{nik} that $\ca \rtimes T^\rho$ is equivalent to the crossed product tensor category $\ca \rtimes G$ constructed by Tambara in \cite{Ta}. 

\begin{rmk} Recall that a fusion category is called \emph{pointed} if all its simple objects are invertible. 
On the other side, a fusion category $\ca$ is called \emph{group-theoretical} if it is Morita equivalent to a 
pointed fusion category, that is, if there exists an exact (hence semisimple) indecomposable $\ca$-module category $\Mo$ such that 
$\End_\ca(\Mo)$ is a pointed fusion category.    See \cite[Subsection 8.8]{ENO1}.

\medbreak The fact that $\ca^G$ is Morita equivalent to the crossed product $\ca \rtimes G$ implies immediately that if  $\ca$ is a pointed fusion category, then any equivariantization $\ca^G$ is group-theoretical, because in this case $\ca \rtimes G$ is itself a pointed fusion category.

We observe, however, that this feature does not extend to more general (even normal) Hopf monads. Take
for instance $H = H_p$ to be the non group-theoretical semisimple Hopf algebra of dimension 
$4p^2$ in \cite[Section 5]{nik},  where $p$ is an odd  prime number. 
It follows from \cite[Proposition 5.2]{nik} that there is an exact sequence of Hopf algebras $$\ku \to \ku\Z_2 \to H \to
\overline H \to \ku,$$ where $\overline H$ is a certain semisimple Hopf algebra introduced by Masuoka. 
Hence we get an exact sequence of fusion categories $$\overline H \rrep \to H\rrep \to \ku^{\Z_2}\rrep.$$ 
Therefore the non group-theoretical fusion category $H\rrep$ is equivalent to the fusion category 
$(\ku^{\Z_2}\rrep)^T$, where $T$ is the normal Hopf monad on the pointed fusion category $\ku^{\Z_2}\rrep$
given by Proposition \ref{monad-ext}.
\end{rmk} 
\end{exa}

Nevertheless we have the following:  

\begin{prop}\label{pointed-comm} Let $\ca$ be a pointed fusion category and $T$ be a normal faithful $\ku$-linear Hopf monad on $\ca$. Suppose that the induced Hopf algebra of $T$ is commutative. Then $\ca^T$ is a group-theoretical fusion category.
\end{prop}

\begin{proof} Since $\ca^T$ is an extension of $\ca$ by $H\rrep$, then $\ca^T$ is an integral fusion category of Frobenius-Perron dimension equal to $\dim H \, |G|$, where $G$ is the group of invertible objects of $\ca$. See \cite[Propositions 4.9 and 4.10]{BN}. By \cite[Corollary 8.14 and Theorem 8.35]{ENO1},  $\ca \rtimes T$ is also an integral fusion category of Frobenius-Perron dimension $\dim H \, |G|$.

Let $H$ be the induced Hopf algebra of $T$. By Proposition \ref{klinear}, we have an equivalence of $\ku$-linear categories $(\ca \rtimes T)^{\op} \simeq H\rrep \boxtimes \, \ca$. 
Since $H$ is commutative, then the category $H\rrep$ is also pointed.  Let $g_1, \dots, g_n$ be the pairwise non-isomorphic invertible objects of $H\rrep$, where $n = \dim H$, and let $h_1, \dot, h_{|G|}$ be the  pairwise non-isomorphic invertible objects of $\ca$. Then the simple objects of $H\rrep \boxtimes \, \ca$ are, up to isomorphism, $g_i \boxtimes h_j$, $1\leq i \leq n$, $1\leq j\leq |G|$. 

In particular the category $H\rrep \boxtimes \, \ca$ has $n|G|$ simple objects. Then so does the category $\ca \rtimes T$. Since $\ca \rtimes T$ is integral, then for every simple object $X \in \ca \rtimes T$, we must have $\operatorname{FPdim} X = 1$, that is, $X$ is an invertible object. Hence the category $\ca \rtimes T$ is pointed and therefore $\ca^T$ is group-theoretical, as claimed. \end{proof}

Proposition \ref{pointed-comm} allows us to recover the fact that any semisimple Hopf algebra $H$ such that $H$ fits into an exact sequence $\ku \to \ku^\Gamma \to H \to \ku F \to \ku$, where $\Gamma$ and $F$ are finite groups, is group-theoretical. In fact, we have in this case $H\rrep \simeq (\ku^\Gamma\rrep)^T$, where $T$ is the associated normal Hopf monad. From Proposition \ref{monad-ext} we have that the induced Hopf algebra of $T$ is the commutative Hopf algebra $\ku^F$.

\end{document}